\documentclass[smallextended]{sn-jnl}
\usepackage[english]{babel}
\usepackage{
    amsmath,
    amssymb,
    amsthm,
    amsfonts,
    mdframed,
    xcolor,
    caption,
    subcaption,
    comment, 
    mathrsfs,
    enumitem,
    float,
    orcidlink,
    graphicx,
    epstopdf,
    multirow,
    textcomp,
    manyfoot,
    booktabs,
    algorithm,
    algorithmicx,
    algpseudocode,
    listings
}
\usepackage[title]{appendix}

\title{A time adaptive optimal control approach for 4D-var data assimilation problems governed by parabolic PDEs}

\author[1]{\fnm{Carmen} \sur{Gr\"a\ss le} \orcidlink{0000-0003-0318-0740}}\email{c.graessle@tu-braunschweig.de}
\author*[1]{\fnm{Jannis} \sur{Marquardt} \orcidlink{0009-0008-0248-7533}}\email{j.marquardt@tu-braunschweig.de}

\affil[1]{\orgdiv{Institute for Partial Differential Equations}, \orgname{TU Braunschweig}, \orgaddress{\street{Universit\"atsplatz 2}, \city{Braunschweig}, \postcode{38106}, \country{Germany}}}

\theoremstyle{plain}
\newtheorem{theorem}{Theorem}[section]
\newtheorem{lemma}[theorem]{Lemma}
\newtheorem{corollary}[theorem]{Corollary}

\newtheorem{remark}[theorem]{Remark}
\newtheorem{definition}[theorem]{Definition}
\newtheorem{example}[theorem]{Example}

\newcommand{\R}{\mathbb{R}}
\newcommand{\N}{\mathbb{N}}
\newcommand{\M}{\mathcal{M}}
\newcommand{\V}{\mathcal{V}}
\renewcommand{\H}{\mathscr{H}}
\newcommand{\C}{\mathcal{C}}
\newcommand{\A}{\mathcal{A}}

\newcommand{\dobs}{{d_{\text{obs}}}}

\newcommand{\X}{\mathcal{X}}
\newcommand{\Y}{\mathcal{Y}}

\renewcommand{\L}{\mathcal{L}}

\newcommand{\intd}{\,\text{d}}
\newcommand{\Uad}{U_{\text{ad}}}

\newcommand{\braket}[2]{
\left\langle #1, #2 \right\rangle
}

\newcommand{\opnorm}[1]{\left|\hspace{-1pt}\left|\hspace{-1pt}\left|#1\right|\hspace{-1pt}\right|\hspace{-1pt}\right|}

\begin{document}

\abstract{
    We interpret the 4D-var data assimilation problem for a parabolic partial differential equation (PDE) in the context of optimal control and revisit the process of deriving optimality conditions for an initial control problem. This is followed by a reformulation of the optimality conditions into an elliptic PDE, which is only dependent on the adjoint state and can therefore be solved directly without the need for e.g.\ gradient methods or related iterative procedures. Furthermore, we derive an a-posteriori error estimation for this system as well as its initial condition. We utilize this estimate to formulate a procedure for the creation of an adaptive grid in time for the adjoint state. This is used for 4D-var data assimilation in order to identify suitable time points to take measurements. 
}

\keywords{Data assimilation, optimal control of partial differential equations, initial control, a-posteriori error estimate, time adaptivity}

\pacs[MSC Classification]{
49J20, 
49K20, 
49M41, 
65M50, 
65M15 
}

\maketitle

\section*{Novelty statement} In this work, we consider a 4D-var data assimilation problem from the perspective of optimal control and present the following contributions:
\begin{itemize}
    \item We consider an optimal control problem where the control enters through the initial condition and reformulate the optimality system as a single second order in time and fourth order in space elliptic problem (similar to e.g.\ \cite{time-adaptivity-mpc,buettner-dissertation,space-time} where the control enters through the right-hand side).
    \item We derive an a-posteriori error estimate in time for the elliptic problem (similar to \cite{time-adaptivity-mpc,space-time}).
    \item We utilize the elliptic formulation in order to solve a 4D-var data assimilation problem and utilize the a-posteriori error estimate to identify suitable time points to take measurements.
\end{itemize}

\section*{Data availability statement}
The implementation of the numerical examples has been realized in Python. The code of the presented examples is available at \url{https://doi.org/10.5281/zenodo.13133804}.

\section{Introduction}
    \subsection{Problem statement and objective}
       We start with a general introduction of the data assimilation problem and compare it to the specific continuous setting considered during the rest of this work. The goal of data assimilation is to update a mathematical model with observations from the real world. In 4D-var data assimilation, the observations $y_i\in \Y_i \subseteq\R^\dobs$ for $0 \leq i\leq N$ are taken at multiple time instances $0 = t_0 <... < t_N = T$, where $T<\infty$. We refer to $[0,T]$ as assimilation window and call $\Y_i$ observation space (at time $t_i$). The current state of the model may be described in the state space by $x_i\in\X\subseteq \R^d$. The forward evolution of the model is governed by some dynamics and may be described by $x_{i+1} = \M_i (x_{i})$ with $\M_i:\X\to\X$. Usually it holds that $\dobs \ll d$. A connection between the state space $\X$ and the observation spaces $\Y_i$ may be realized by observation operators $\mathcal H_i:\X\to\Y_i$. 
        In order to match the model prediction with the observations, the initial state $x_0$ can be chosen as a solution to 
        \begin{subequations}\label{eq:da-problem}
        \begin{equation}
            \underset{x_0\in\R^d}{\mathrm{argmin}} \Big\{J(x_0):=\; \frac{1}{2}\sum_{i=0}^N \left\Vert \mathcal H_i(x_i) - y_i \right\Vert^2_{\Y_i} + \frac{\alpha}{2}\left\Vert x_0 - x_0^{(b)}\right\Vert^2_\X\Big\},
        \end{equation}
        subject to
        \begin{equation}
            x_{i+1} =\; \M_i (x_i) \quad \forall i \in\{0,...,N-1\},
        \end{equation}
        \end{subequations}
        with an initial guess $x^{(b)}_0 \in \mathbb{R}^d$ and a trust coefficient $\alpha >0$ which describes how much confidence can be put into $x^{(b)}_0$ compared to the measurements $y_i$. 

        It is convenient to rewrite $J(x_0)$ in some situations, such that the operators and vectors are not only specified on the space, but in a space-time domain. Following the lines in \cite{data-freitag, synergy-inverse-freitag}, such an \glqq all at once approach\grqq\space can be formulated by writing $J(x_0)$ as 
        \begin{equation}
            J(x_0)= \frac{1}{2} \left\Vert \mathcal H(x) - y^{(d)} \right\Vert^2_{ \Y} + \frac{\alpha}{2}\left\Vert x_0 - x_0^{(b)}\right\Vert^2_{\X},
        \end{equation}
        where $\Y = \Y_0\times \cdots\times \Y_N$, $x = (x_0^T, \hdots,x_N^T)^T$, $\mathcal H (x) = (\mathcal H(x_0)^T,\hdots,\mathcal H(x_N)^T)^T$ and $y^{(d)} = (y_0^T, \hdots,y_N^T)^T$. Hence, the occurring states and observations are now defined on a discrete space-time grid instead of multiple spatial domains as before.
        
        It is known that the data assimilation problem may also be formulated in the context of optimal control, see e.g.\ \cite{data-assimilation-book, fredi}. If the governing dynamics is linear and parabolic, the data assimilation task may be formulated as
        \begin{subequations}\label{eq:control-problem}
            \begin{equation}
                \underset{v\in \Uad}{\mathrm{argmin}} \Big\{J(v) := \; \frac{1}{2}\left\Vert y(v) - y^{(d)}\right\Vert^2_{L^2(\Omega_T)} + \frac{\alpha}{2}\left\Vert v - y^{(b)}\right\Vert^2_{L^2(\Omega)}\Big\},
            \end{equation}
            subject to
            \begin{equation}
            \left\{
                \begin{array}{rll}
                    y_{t}(t,x; v) + \A y(t,x; v) &=\; f(t,x) &\text{for } (t,x)\in \Omega_T,\\
                    y(t,x; v) &=\; 0 & \text{for } (t,x )\in\Sigma_T,\\
                    y(0,x; v) &=\; v & \text{for } x\in \Omega
                \end{array}
            \right.
            \end{equation}
        \end{subequations}
         with $f, y^{(d)}\in L^2(\Omega_T)$, $y^{(b)}\in L^2(\Omega)$ and the set of admissible controls $\Uad\subseteq L^2(\Omega)$. Here, $\Omega\subseteq \R^d$ denotes an open and bounded domain for which we assume that it has a sufficiently smooth boundary $\Gamma := \partial \Omega$. Furthermore, we use the notation $\Omega_T := (0,T] \times \Omega$ as well as $\Sigma_T :=(0,T] \times \Gamma$. The self-adjoint second order elliptic differential operator $\A: L^2(\Omega_T)\to L^2(\Omega_T)$ acts on the space domain and can be written as
         \begin{equation}\label{eq:introduction-A}
            \A y(t,x;v) = - \sum_{i,j=1}^d \frac{\partial}{\partial x_i}\left(a_{ij}(x) \frac{\partial}{\partial x_j}y(t,x;v)\right) + a_0(x)y(t,x;v),
         \end{equation}
         with $a_{ij}, a_0 \in L^\infty (\Omega)$ such that $a_{ij}=a_{ji}$.
         
         While passing from the discrete context in \eqref{eq:da-problem} to the continuous problem \eqref{eq:control-problem}, it has been assumed that the data $y^{(d)}$ is available everywhere, i.e.\ $y^{(d)}\in L^2(\Omega_T)$. Despite our intention of including more freedom to the measurements in future works, the absence of the observation operator $\mathcal H$ in the continuous problem is required in the context of this paper.\\

         The first objective of this paper is to show equivalence between finding a solution to the initial control problem \eqref{eq:control-problem} and solving a certain elliptic PDE, which is second order in time and fourth order in space. The solution to this PDE allows to determine the optimal control to \eqref{eq:control-problem} and hence an updated initial state in the data assimilation problem \eqref{eq:da-problem} readily.\\

        Choosing times $0=t_0 < ... < t_N=T$ at which data is incorporated into the model in an equidistant way might not be favorable in some situations. While choosing $N$ small may lead to a grid on which instantaneous changes of the reality may not be reflected by the model, a large amount of observations may prevent real-time or quasi-real-time applications. The fact that in some scenarios each measurement may cause costs or that the process of taking a measurement may influence the systems behavior are further motivations for keeping $N$ small. We will therefore investigate an a-posteriori error estimate based procedure of finding optimal times at which observations should be taken.\\

        In order to accomplish the reformulation and the error estimation, we will closely follow the steps in \cite{time-adaptivity-mpc, space-time}, where similar results have been established for different control problems. We account for further works which investigated similar reformulations in Section~\ref{sec:related-work}.

        \subsection{Paper organization}
        While Section~\ref{sec:optimal-control-problem} does not contain novel results, it introduces the systems and notations needed for the rest of the paper and recalls the required results from optimal control theory on which the subsequent considerations in this paper built upon. The key outcome of this section is Theorem~\ref{thm:optimality-system}, in which existence and uniqueness of a solution to the optimal control problem is proven and the optimality conditions are derived. 

        Next, the connection to the fourth order elliptic equation will be established in Section~\ref{sec:elliptic_fourth_order}. After the derivation of its weak form and an alternative mixed weak formulation, numerical examples follow, in which the effectiveness of the derived method will be demonstrated by the determination of a solution to the fourth order system using the finite element method in space-time. 

        Section~\ref{sec:time-adaptivity} starts with the derivation of the a-posteriori error estimate for the discretization of the fourth order elliptic problem from before. An adaptive method for the refinement of the time intervals using the a-posteriori error estimate will be proposed in this section. The section will be concluded with numerical examples.

        \subsection{Related work}\label{sec:related-work}
        There is a wide range of literature on which this work builds up or is related to. Data assimilation has a wide area of applications and occurs in geosciences \cite{data-geo-1, data-geo-2, data-geo-3}, medicine \cite{data-med-1, data-med-2}, meteorology \cite{data-mete-1, data-mete-2, data-mete-3, data-mete-4} and many more. It is typically introduced as an inverse Bayesian problem, see i.e.\ \cite{data-freitag, data-assimilation-kody, bayesian-tutorial}. An introduction to optimal control of partial differential equations may be found for instance in \cite{hinze-ulbrich,lions, manzoni21, fredi}. The control theory in Section~\ref{sec:optimal-control-problem} is mostly oriented after \cite{hinze-ulbrich,lions, fredi} and strongly utilizes the notation which is established in \cite{lions}. It is the commonly used notation for parabolic equations in Hilbert spaces, for which the reader may refer to e.g.\ \cite{lions72, showalter96, zeidler} for further reference.
        The idea of using the framework of optimal control for data assimilation is well established and we refer to \cite{data-assimilation-book, fisher09} for more insight.

        The approach of combining the state and adjoint state into a single equation has already proven to be successful in case of control problems with distributed control \cite{time-adaptivity-mpc, buettner-dissertation, space-time, nps-smooth-regularization, nps-comsol, nps-parabolic-optimal}. The derivation of an a-posteriori error estimate has also been established and used for the creation of a time adaptive grid in the context of distributed control \cite{time-adaptivity-mpc, alla18, space-time}. For more general information about a-posteriori error estimates, we refer to \cite{interpolation-result}.
        
    \section{Framework for optimal control via initial condition}\label{sec:optimal-control-problem}
    In this section the optimal control framework for initial control problems of the form \eqref{eq:control-problem} will be set up. The stated results are well known to readers experienced in optimal control. Nevertheless, this section provides proofs for some of those results which are often dismissed as an exercise in the standard literature. The notation and procedure is oriented after \cite{hinze-ulbrich,lions, fredi}.
    
    \subsection{Notation and governing dynamics}\label{sec:governing-dynamics}

    Let $(V, H, V^*)$ be an evolution triple. Hence, $H$ and $V$ are separable Hilbert spaces with dual spaces $V^*$ and $H^*$, such that $V$ is densely and continuously embedded in $H$. Furthermore, it holds $V \hookrightarrow H=H^* \hookrightarrow V^*,$ where the usual identification of $H$ with its dual has been applied.
    
    The inner product in $V$ is given as $\braket{v}{w}_V$, while $\braket{u^*}{v}_{V^*, V} := u^*(v)$ denotes the dual pairing between $V^*$ and $V$ for $v,w\in V$ and $u^*\in V^*$. Since we identified $H^*$ with $H$, it holds that $\braket{u}{v}_H =\braket{u}{v}_{H^*, H}$ for all $u,v\in H = H^*$.

    The space of (equivalence classes of) measurable square integrable functions $y:\mathfrak D\to \mathfrak C$ is given as $L^2(\mathfrak D; \mathfrak C)$. The case of $\mathfrak C = \R$ will be denoted as $ L^2(\mathfrak D) := L^2(\mathfrak D; \R)$. We introduce the abbreviation $\V := L^2(0,T; V)$ and consider $\V^* := L^2(0,T; V^*)$ as its dual. By denoting $y_t$ as the derivative of $y\in\V$, we may define the Bochner space with its norm as
    $$W(0,T; V) := \left\{y \in \V \;\middle| \;y_t \in \V^*\right\},\quad\Vert y\Vert_{W(0,T;V)} := \left(\Vert y\Vert^2_{\V} + \Vert y_t \Vert^2_{\V^*}\right)^{\frac{1}{2}}.$$

    Let us introduce the general setting for which the dynamics will be considered in this section by defining a bilinear form $a:(0,T]\times V^2\to\R$, which fulfills
    \begin{itemize}
        \item $a(t; \cdot, \cdot)$ is continuous for fixed $t\in (0,T]$, i.e.\ there exists $c>0$ such that
            \begin{equation}\label{eq:elliptic-one}
            \vert a(t;\varphi,\psi)\vert \leq c \Vert \varphi\Vert_V\cdot\Vert\psi\Vert_V\quad \forall \varphi,\psi\in V,
            \end{equation}
        \item $a(t; \cdot, \cdot)$ is $V$-elliptic for fixed $t\in (0,T]$, i.e.\ there exists $\gamma > 0$ such that
            \begin{equation}\label{eq:elliptic-two}
                a(t;\psi,\psi) \geq \gamma \Vert\psi\Vert^2_V\quad \forall \psi\in V,
            \end{equation}
        \item $a(\,\cdot\,; \varphi, \psi)$ is measurable on $(0,T]$ for all $\varphi,\psi\in V$.
    \end{itemize}
     
    For each $t\in(0,T]$ we define an elliptic operator $\A(t)$ with dual operator $\A^*(t)$, which is based upon $a(t;\varphi,\psi)$, via
    $$a(t;\varphi,\psi) = \braket{\A(t)\varphi}{\psi}_{V^*, V} = \braket{\A^*(t) \psi}{\varphi}_{V^*,V}\quad \forall \varphi,\psi \in V,$$
    where $\A(t)$ and $\A^*(t)$ are bounded and linear, i.e.\ $\A(t),\A^*(t)\in \L(\V, \V^*)$. The Cauchy problem: find $y\in W(0,T; V)$ such that
    \begin{equation}\label{eq:evolution}
        \left\{
        \begin{array}{rll}
            y_t + \A(t) y &=\; f & \text{in } \V^*,\\
            y(0) &=\; y_0& \text{in } H
        \end{array}
        \right.
    \end{equation}
    admits a unique solution in $W(0,T; V)$, which depends continuously on $f $ and $y_0$ (see e.g.\ \cite[Chapter III, Theorem 1.2]{lions}).

    \subsection{Initial control problem}\label{sec:initial-control}
    Let $\H$ denote an arbitrary Hilbert space with dual space $\H^*$. In this section, we study the optimal control problem 
    \begin{subequations}
    \label{eq:abstract-control}
        \begin{equation}\label{eq:control-min-functional}
            \underset{v\in \Uad}{\mathrm{argmin}} \Big\{J(v) := \; \frac{1}{2}\left\Vert C y(v) - y^{(d)}\right\Vert^2_\H + \frac{\alpha}{2}\left\Vert v - y^{(b)}\right\Vert^2_H\Big\},
        \end{equation}
        subject to
        \begin{equation}\label{eq:control-state}
            \left\{
            \begin{array}{rll}
                y_t(v) + \A(t) y(v) &=\; f & \text{in }\V^*,\\
                y(0;v) &=\; v & \text{in }H,\\
                y(v) &\in\; W(0,T;V),
            \end{array}
            \right.
        \end{equation}
        where $\alpha > 0$, $y^{(d)}\in L^2(0,T;\H)$, $y^{(b)}\in H$ and $\A(t)\in\L(\V, \V^*)$ as defined in Section~\ref{sec:governing-dynamics}. The space of controls has been chosen to be $H$, including the closed and convex subset of admissible controls $\Uad \subseteq H$. Furthermore, we use $C\in\L(\V,\H)$ in order to connect the state space $\V$ and the observation space $\H$ in case that state and observation space differ. Despite the possibility of allowing the state and observation space to vary from each other using $C$, we will not pursue this option further in this work and will only use $C$ to express an isomorphism between $\V$ and $\H$ as soon as we concretize the setting in Section \ref{sec:elliptic_fourth_order}. 
    
        The adjoint equation is given as
        \begin{equation}\label{eq:control-adjoint}
            \left\{
            \begin{array}{rll}
                - p_t(v) + \A^*(t) p(v) &=\; C^*\Lambda (C y(v) - y^{(d)}) & \text{in }\V^*,\\
                p(T;v) &=\; 0 & \text{in } H,\\
                p(v) &\in\; W(0,T;V),
            \end{array}
            \right.
        \end{equation}
    \end{subequations}
    with $C^*\in\L (\H^*, \V^*)$ being the adjoint of $C$ and the canonical isomorphism $\Lambda: \H \to \H^*$ provided by the Riesz representation theorem (we do not identify $\H$ with $\H^*$ yet, since this results in difficulties in case of $\H \subseteq V^*$ after the preceded identification of $H$ with its dual). Slightly different versions of the next Theorem are well-known, see e.g.\ \cite{lions, fredi}.\\
    
    \begin{theorem}[Optimality system]\label{thm:optimality-system}
        The optimal control problem \eqref{eq:abstract-control} admits a unique solution $ u\in\Uad$ which satisfies
        $$\braket{p(0,u) + \alpha ( u - y^{(b)})}{v - u}_H \geq 0 \quad \forall v \in \Uad.$$
    \end{theorem}
    \begin{proof}
        As described in Section~\ref{sec:governing-dynamics}, both \eqref{eq:control-state} and \eqref{eq:control-adjoint} admit a unique solution in $W(0,T;V)$ for each control $v\in \Uad$ (with the adjoint equation going backwards in time). Furthermore, there exists a linear and continuous operator, which maps the initial condition of the initial value problem onto its solution in space-time. We will denote the concatenation of this operator and $C$ by $S:H \to \H$. Hence, $S$ maps $y(0;v)\mapsto Cy(v)$ and is linear and continuous, too. By means of $S$ we can equivalently consider
        \begin{equation}\label{eq:altered-control-problem}
            \underset{\tilde v\in \tilde U_\text{ad}}{\mathrm{argmin}} \Big\{\tilde J(\tilde v) := \; \frac{1}{2}\left\Vert  S \tilde v - \left(y^{(d)} - S y^{(b)}\right)\right\Vert^2_\H + \frac{\alpha}{2}\left\Vert \tilde v\right\Vert^2_H\Big\}
        \end{equation}
        instead of \eqref{eq:control-min-functional}-\eqref{eq:control-state} for the proof of existence and uniqueness of an optimal control. In the new formulation, the variable $\tilde v$ arose from the substitution of $\tilde v := v - y^{(b)}$, followed by the switch of $\Uad$ with $\tilde U_\text{ad} := \{v - y^{(b)}\mid v \in \Uad\}$. The existence and uniqueness of an optimal control $\tilde u \in \tilde U_\text{ad}$ of \eqref{eq:altered-control-problem} and hence an optimal control $u\in\Uad$ of \eqref{eq:abstract-control} follows from \cite[Theorem 2.14]{fredi}.

        An admissible control $u\in\Uad$ is optimal for \eqref{eq:control-min-functional} if $\braket{ J'(u)}{ v - u}_{H^*,H} \geq 0$ for all $v\in\Uad$. Executing the derivative holds the equivalence to
        \begin{equation}\label{eq:variational-inequality}
        \braket{ C^*\Lambda\left(C y(u) - y^{(d)}\right)}{ y(v) - y(u)}_{\V^*,\V} + \alpha \braket{  u - y^{(b)}}{v -  u}_H\geq 0
        \end{equation}
        (cf. \cite[Chapter I, Section 1.4]{lions}). The utilization of the end time value of the adjoint allows to write
        \begin{equation*}
            \begin{split}
                \int_0^T&\braket{-p_t(u)}{y(v) -y(u)}_{V^*, V}\intd t\\
                &= \int_0^T \braket{ y_t(v) - y_t(u)}{p(u)}_{V^*, V}\intd t- \Big[\braket{ p(u)}{ y(v)-y(u)}_V\Big]_{t=0}^T \\
                &= \int_0^T \braket{y_t(v) - y_t(u)}{p(u)}_{V^*, V} \intd t + \braket{p(0;u)}{ v- u}_H 
            \end{split}.
        \end{equation*}
        Furthermore, it holds that
        \begin{equation*}
            \int_0^T\braket{ \A^*(t)p(u)}{ y(v)-y(u)}_{V^*,V} \intd t = \int_0^T \braket{ p(u)}{ \A(t) \left(y(v) - y(u)\right)}_{V,V^*}\intd t.
        \end{equation*}
        The combination of these equations with the adjoint equation and \eqref{eq:variational-inequality} yields
        \begin{equation*}
                 \int_0^T \braket{\left[\frac{\partial}{\partial t} + \A(t) \right](y(v)-y( u))}{ p(u)}_{V^*,V} \intd t + \braket{ p(0; u) + \alpha (u-y^{(b)})}{v -u }_H \geq 0.
        \end{equation*}
        Since the first term vanishes after the insertion of the state equation, the claim follows.\\
    \end{proof}

    \section{Data assimilation via reformulation of the optimality system}
    \label{sec:elliptic_fourth_order}
    In this section the optimality conditions of the initial control problem \eqref{eq:control-problem} will be reformulated as an elliptic system which is second order in time and fourth order in space. After the introduction of the notation in Section~\ref{sec:sobolev-notions}, the fourth order system will be derived in Section~\ref{sec:reformulation}. The existence and uniqueness of the fourth order system will be discussed in Section~\ref{sec:existence-and-uniqueness}, while Section~\ref{sec:mixed-form} provides a mixed variational form of the problem, which will then be numerically implemented with piecewise linear, continuous finite elements. We will discuss the results of those experiments in Section~\ref{sec:numerical-examples-1}.
    
    The idea and structure of this section found its motivation in \cite{space-time, time-adaptivity-mpc}. Nevertheless, similar reformulations have been done for slightly different problems by other authors, whose works have been accounted for in Section~\ref{sec:related-work}.
    
    \subsection{Notations for Sobolev spaces}\label{sec:sobolev-notions}
    First, let us introduce the Sobolev spaces of order $k\in\N_0$ as
    \begin{equation*}
        \begin{split}
            H^k(\mathfrak D, \mathfrak C) &:= \left\{v\in L^2(\mathfrak D, \mathfrak C)\mid v \text{ has weak derivatives } D^\beta v \in L^2(\mathfrak D, \mathfrak C)\text{ for all }|\beta|\leq k\right\},\\
            H^k_0(\mathfrak D, \mathfrak C) &= \{v\in H^k(\mathfrak D, \mathfrak C)\mid D^\beta v = 0 \text{ on }\partial \Omega\text{ in the sense of traces } (|\beta| \leq k-1)\}
        \end{split}
    \end{equation*}
    with norm $\Vert v\Vert_{k,\Omega} := \left(\sum_{|\beta|\leq k}\Vert D^\beta v \Vert_{L^2(\mathfrak D, \mathfrak C)}^2\right)^{\frac{1}{2}}$. We write $H^k(\mathfrak D) := H^k(\mathfrak D, \R)$, $H^k_0(\mathfrak D) := H^k_0(\mathfrak D, \R)$ and introduce the notion of $H^{-1}(\Omega)$ as the dual space of $H_0^1(\Omega)$. Furthermore, we define the spaces
    \begin{equation*}
        \begin{split}
            H^{2,1}(\Omega_T) &:= L^2(0,T;H^2(\Omega)\cap H_0^1(\Omega))\cap H^1(0,T;L^2(\Omega)),\\
            H^{2,1}_0(\Omega_T) &:= \left\{y \in H^{2,1}(\Omega_T)\mid y(T) = 0 \text{ in }\Omega\right\}.
        \end{split}
    \end{equation*}
    A norm for both $H^{2,1}(\Omega_T)$ and $H^{2,1}_0(\Omega_T)$ is given as
    $$\Vert v\Vert_{2,1,\Omega_T} := \left(\Vert v \Vert^2_{L^2(0,T;H^2(\Omega))} + \Vert v \Vert^2_{H^1(0,T;L^2(\Omega))}\right)^{\frac{1}{2}}.$$

    \subsection{Reformulation of the optimality system}
    \label{sec:reformulation}
    We recognize \eqref{eq:control-problem} as a special case of \eqref{eq:abstract-control} after choosing $V= H^1_0(\Omega)$, $H=L^2(\Omega)$, $\H = L^2(\Omega_T)$ and $C$ as the injection map $C:L^2(0,T; L^2(\Omega))\to L^2(\Omega_T)$. We identify $\H$ with its dual, such that the canonical isomorphism $\Lambda$ may be neglected from now on. Furthermore, we choose $\A(t)$ according to \eqref{eq:introduction-A} independent of time as $\A: L^2(0,T;H^1_0(\Omega))\to L^2(0,T;H^{-1}(\Omega))$ with
    \begin{equation}\label{eq:operator-A}
        \A v = - \nabla \cdot \left[A \nabla v\right] + a_0 v = - \sum_{i,j=1}^d \frac{\partial}{\partial x_i}\left(a_{ij} \frac{\partial}{\partial x_j}v\right) + a_0v,
    \end{equation}
    where the matrix $A = (a_{ij})_{1\leq i,j\leq d}$ is symmetric with coefficients $a_{ij}=a_{ji}$ and $a_{ij}, a_0 \in L^\infty (\Omega)$. In comparison to Section~\ref{sec:optimal-control-problem}, $\A(t)\equiv \A$ is now independent of time and due to the symmetry of $a_{ij}$ self-adjoint in $H^1_0(\Omega)$ such that $\A^*=\A$. Since $\A$ is elliptic, equations \eqref{eq:elliptic-one} and \eqref{eq:elliptic-two} hold true in the form: there exist  $c,\gamma>0$ such that it holds for all $v,w\in H^1_0(\Omega)$ that
    \begin{equation}\label{eq:elliptic-continuous}
        \left\vert \int_\Omega (\A v) w\intd x\right\vert  \leq c \Vert v\Vert_{1,\Omega} \Vert w \Vert_{1,\Omega}
    \end{equation}
    and
    \begin{equation}\label{eq:elliptic-coercive}
        \int_\Omega (\A v) v\intd x \geq \gamma  \Vert v \Vert_{1,\Omega}^2.
    \end{equation}
    
    Existence of an optimal $u\in\Uad$ is ensured by Theorem~\ref{thm:optimality-system}. Moreover, there exists a unique state $y\in W(0,T;H^1_0(\Omega))$ and a unique adjoint state $p\in W(0,T; H^1_0(\Omega))$ which obey the state equation
    \begin{subequations}\label{eq:concrete-conditions}
        \begin{equation}\label{eq:concrete-state}
        \left\{
            \begin{array}{rll}
                 y_t + \A y &=\; f &\text{in } \Omega_T,\\
                 y &=\; 0 & \text{on } \Sigma_T,\\
                 y(0) &=\; u & \text{in } \Omega,
            \end{array}
            \right.
        \end{equation}
        the adjoint state equation
        \begin{equation}\label{eq:concrete-adjoint}
            \left\{
            \begin{array}{rll}
                - p_t + \A p &=\; y - y^{(d)} &\text{in }\Omega_T,\\
                  p &=\; 0 &\text{on } \Sigma_T,\\
                  p(T) &=\; 0 &\text{in } \Omega
            \end{array}
            \right.
        \end{equation}
        and the optimality condition
        \begin{equation}\label{eq:projected-variational}
            u = \mathbb P_{\Uad} \left[y^{(b)}-\frac{1}{\alpha}p(0)\right] \quad \text{in } \Omega.
        \end{equation}
    \end{subequations}

    In order to rewrite this optimality condition as a single fourth order elliptic problem, we need the following regularity result.\\

    \begin{lemma}[{Higher regularity, after \cite[\S 7.1.3, Theorems 5 \& 6]{evans}}]\label{lem:regularity}
        Let $y$ denote a solution to \eqref{eq:concrete-state} and $p$ of \eqref{eq:concrete-adjoint}.
        \begin{enumerate}[label=(\roman*)]
            \item Assume $u\in H_0^1(\Omega)$ and $f, y_d \in L^2(0,T; L^2(\Omega))$. Then it holds
            $$y,p\in L^2(0,T;H^2(\Omega))\cap L^\infty(0,T; H_0^1(\Omega))\cap H^1(0,T;L^2(\Omega)).$$

            \item Assume $u\in H_0^1(\Omega)\cap H^3(\Omega)$, $f, y^{(d)}\in L^2(0,T; H^2(\Omega))\cap H^1(0,T; L^2(\Omega))$ and let the first order compatibility condition $f(0) - \A u\in H_0^1(\Omega)$ hold true. Then it holds
            $$y,p\in L^2 (0,T; H^4(\Omega))\cap H^1(0,T; H^2(\Omega))\cap H^2(0,T; L^2(\Omega))$$
        \end{enumerate}
    \end{lemma}
    Following along the lines in \cite{time-adaptivity-mpc, space-time}, the optimality system \eqref{eq:concrete-conditions} can be reformulated, such that only the adjoint state variable $p$ is present in the resulting fourth order elliptic boundary value problem, while $y$ and $u$ are eliminated.\\

    \begin{theorem}[Fourth order elliptic problem]
        Let the state variable $y$ and the adjoint state $p$ fulfill $y,p\in W(0,T; H_0^1(\Omega))$. Assume that the control problem is unconstrained and that the optimal control fulfills $u\in H^1_0(\Omega) \cap H^3(\Omega)$. Furthermore, let $f,y^{(d)}\in L^2(0,T; H^2(\Omega))\cap H^1(0,T; L^2(\Omega))$ and let the first order compatibility condition $f(0) - \A u\in H_0^1(\Omega)$ hold true. Then $p$ forms a (strong) solution to 
        \begin{equation}\label{eq:fourth-order-system}
            \left\{
            \begin{array}{rll}
                -p_{tt} + \A^2p &=\; f - y_t^{(d)} - \A y^{(d)} & \text{in }\Omega_T,\\
                \A p &=\; -y^{(d)} & \text{on } \Sigma_T,\\
                p &=\; 0 &\text{on } \Sigma_T,\\
                \left(-p_t+\A p +\frac{1}{\alpha}p\right)(0) &=\;  y^{(b)} - y^{(d)}(0)& \text{in } \Omega,\\
                p(T) &=\; 0 & \text{in } \Omega
            \end{array}
            \right.
        \end{equation}
        a.e.\ in space-time.
    \end{theorem}

    \begin{proof}
        First notice that all requirements for Lemma~\ref{lem:regularity} (ii) are fulfilled. Hence, the equations may not only be considered in their weak forms, but a.e.\ in space-time. Furthermore, the regularity for the following steps is provided.
        The state and adjoint equations are given as
        $$y_t + \A y = f \quad\text{and}\quad -p_t + \A p = y - y^{(d)}.$$
        Inserting the adjoint equation into the state equation gives
        $$\left[\frac{\partial}{\partial t} + \A\right]\left(-p_t + \A p + y^{(d)}\right) = f.$$
        Utilizing the interchangeability of the time derivative with $\A$, this reduces to
        $$-p_{tt} + \A^2 p = f - y_t^{(d)}-\A y^{(d)},$$
        which is shows the claim in $\Omega_T$. We know from the adjoint equation that
        $$p = 0\text{ on }\Sigma_T,\quad p(T) = 0 \text{ in }\Omega.$$
        The boundary condition $p = 0$ also implies $p_t =0$ on $\Sigma_T$. Hence, if we insert this and the fact that $y=0$ on $\Sigma_T$ into the adjoint equation, the boundary condition
        $$\A p = - y^{(d)}\quad \text{on }\Sigma_T$$
        arises. The initial condition $y(0) = u$ from the state system can be used in combination with the adjoint equation and \eqref{eq:projected-variational} with $\Uad = L^2(\Omega)$ to get
        $$\left[- \frac{\partial}{\partial t} + \A\right]p(0) = u -y^{(d)}(0) = y^{(b)}-\frac{1}{\alpha}p(0) - y^{(d)}(0),$$
        which gives the remaining initial condition.
    \end{proof}

    \begin{remark}[Constrained control]
        It is possible to consider a constrained control $u\in\Uad$. In this case the initial condition in \eqref{eq:fourth-order-system} becomes
        $$-p_t(0) + \A p(0) = \mathbb P_{\Uad}\left[y^{(b)} - \frac{1}{\alpha}p(0)\right]-y^{(d)}(0)\quad\text{in }\Omega.$$
        For the specific example of
        $$\Uad = \{u\in L^2(\Omega) \mid \underline u \leq u \leq \overline u \text{ a.e. in }\Omega,\text{ with }\underline u,\overline u \in L^\infty (\Omega)\},$$
        the projection can be written as
        $$\mathbb P_{\Uad}[g]= \begin{cases}
            g(x) & \text{if } \underline u (x) \leq g(x) \leq \overline u(x),\\
            \underline u (x) & \text{if } g(x) < \underline u(x),\\
            \overline u (x) & \text{if } g(x) > \overline{u}(x).
        \end{cases}$$
    \end{remark}

    Let us homogenize \eqref{eq:fourth-order-system} in order to gain a simplified notation in the following sections. Hence, let $g$ be a function which is sufficiently smooth and fulfills the boundary conditions as well as initial and end time conditions of \eqref{eq:fourth-order-system}. For a solution $p$ of \eqref{eq:fourth-order-system}, we may define $\tilde p := p - g$ and arrive at
    \begin{equation}\label{eq:homogenized-fourth-order}
        \left\{
        \begin{array}{rll}
            - \tilde p_{tt} + \A^2\tilde p &=\; \tilde f & \text{in }\Omega_T,\\
            \A \tilde p &=\; 0 & \text{on } \Sigma_T,\\
            \tilde p &=\; 0 &\text{on } \Sigma_T,\\
            \left(- \tilde p_t + \A\tilde p + \frac{1}{\alpha}\tilde p\right)(0) &=\; 0& \text{in } \Omega,\\
            \tilde p(T) &=\; 0 & \text{in } \Omega,
        \end{array}
        \right.
    \end{equation}
    where $\tilde f \in L^2(0,T;H^{-1}(\Omega))$ is given as
    \begin{equation}\label{eq:hom_rhs}
        \tilde f := f - y_t^{(d)} - \A y^{(d)} + g_{tt} - \A^2g.
    \end{equation}

    \subsection{Existence and uniqueness of a weak solution}\label{sec:existence-and-uniqueness}
    For the derivation of the weak formulation of \eqref{eq:homogenized-fourth-order}, we introduce the symmetric bilinear form $B:H_0^{2,1}(\Omega_T)\times H_0^{2,1}(\Omega_T)\to \R$ as
    \begin{equation*}
        \begin{split}
            B(v,w) := \int_{\Omega_T} &v_tw_t + \A v \A w \intd x\intd t
            + \int_\Omega A \nabla v(0) \cdot \nabla w(0) + \left(a_0 + \frac{1}{\alpha}\right) v(0) w(0) \intd x
        \end{split}
    \end{equation*}
    and the linear form $L:H^{2,1}_0(\Omega_T)\to \R$ as
    $$L(w) := \int_{\Omega_T} \tilde fw\intd x\intd t,$$
    where $\tilde f \in L^2(0,T; H^{-1}(\Omega))$ is given as in \eqref{eq:hom_rhs}.\\

    \begin{definition}[Weak formulation]
        The weak formulation of equation \eqref{eq:homogenized-fourth-order} is given by: find $\tilde p\in H^{2,1}_0(\Omega_T)$, such that 
        \begin{equation}\label{eq:weak_normal}
            B(\tilde p, w) = L(w)\quad\forall w\in H^{2,1}_0(\Omega_T).
        \end{equation}
    \end{definition}

    We start with the proof that a solution to \eqref{eq:weak_normal} is a weak solution to \eqref{eq:fourth-order-system} and vice versa that the homogenized solution to \eqref{eq:fourth-order-system} satisfies \eqref{eq:weak_normal}.\\

    \begin{theorem}\label{thm:equivalence-normal}
        Let $p$ denote a solution to \eqref{eq:fourth-order-system} and let $g$ be a sufficiently smooth function which fulfills the boundary, initial and end time conditions of \eqref{eq:fourth-order-system}. Then $\tilde p = p - g$ is a solution to the weak formulation \eqref{eq:weak_normal}. If on the other hand a weak solution $\tilde p$ of \eqref{eq:weak_normal} additionally fulfills the assumptions of Lemma~\ref{lem:regularity} (ii), then $p = \tilde p - g$ satisfies \eqref{eq:fourth-order-system}.
    \end{theorem}

    \begin{proof}
        Let $p$ be a solution to \eqref{eq:fourth-order-system} and $g$ a sufficiently smooth function which fulfills the boundary, initial and and end time conditions of \eqref{eq:fourth-order-system}. Then $\tilde p = p -g$ forms a solution to the homogenized equation \eqref{eq:homogenized-fourth-order}. We test $-\tilde p_{tt} + \A^2\tilde p$ by $w\in H^{2,1}_0(\Omega_T)$ and apply integration by parts in time direction. Since $\A$ is self-adjoint, this yields
        $$\int_{\Omega_T} - \tilde p_{tt} w + \A^2 \tilde p w\intd x \intd t = \int_{\Omega_T} \tilde p_t w_t + \A \tilde p \A w \intd x \intd t - \int_\Omega\tilde p_t(T) w(T) - \tilde p_t(0)w(0)\intd x.$$
        We use the facts that $w(T) = 0$ and $\tilde p_t(0) = \A\tilde p(0) + \frac{1}{\alpha}\tilde p(0)$ in $\Omega$, in combination with Green's identity, in order to derive that
        \begin{equation*}
            \begin{split}
                \int_\Omega \tilde p_t(T)& w(T) - \tilde p_t(0)w(0)\intd x = -\int_\Omega \A\tilde p(0) w(0) +\frac{1}{\alpha}\tilde p(0)w(0)\intd x\\
                \hfill &= \int_\Omega \nabla \cdot \left[A \nabla \tilde p(0)\right]w(0) - \left( a_0 +\frac{1}{\alpha}\right) \tilde p(0)w(0)\intd x\\
                \hfill &= - \int_\Omega A \nabla \tilde p(0)\cdot \nabla w(0) + \left( a_0 +\frac{1}{\alpha}\right) \tilde p(0)w(0)\intd x + \int_{\partial \Omega} [A\nabla \tilde p(0) \cdot n] w(0) \intd S.
            \end{split}
        \end{equation*}
        Since $w(0) = 0$ on $\partial \Omega$, the integral over the boundary vanishes. Reinserting the second in the first equation then yields
        \begin{equation*}
            \int_{\Omega_T} - \tilde p_{tt} w + \A^2 \tilde p w\intd x \intd t = B(\tilde p, w).
        \end{equation*}
        Testing the right hand side of the homogenized equation \eqref{eq:homogenized-fourth-order} with $w\in H^{2,1}_0(\Omega_T)$ gives directly
        $$\int_{\Omega_T}\tilde f w\intd x\intd t = L(w).$$
        While this concludes the proof of the forward direction, the backward direction can be proven by tracing back the steps above, utilizing the additionally required regularity.
    \end{proof}

    Notice, that Theorem~\ref{thm:equivalence-normal} neither implies existence nor uniqueness of a solution to \eqref{eq:weak_normal} since the strong solution \eqref{eq:fourth-order-system} may not exist under the regularity assumptions of \eqref{eq:weak_normal}. The fact that a unique solution to \eqref{eq:weak_normal} exists anyway will be proven in the remaining part of this section. In order to do so, consider the norm
    \begin{equation}\label{eq:opnorm}
        \opnorm{v}:=\left(\Vert v_t\Vert_{0,\Omega_T}^2 + \left\Vert \A v\right\Vert_{0,\Omega_T}^2\right)^{\frac{1}{2}},
    \end{equation}
    which is equivalent to $\Vert\cdot \Vert_{2,1,\Omega_T}$ on $H_0^{2,1}(\Omega_T)$.\\
    
    \begin{lemma}\label{lem:equivalence-norms}
        There exist positive constants $c_1,c_2>0$ such that
        $$c_1 \opnorm{v} \leq \Vert v\Vert_{2,1,\Omega_T} \leq c_2 \opnorm{v}\quad \forall v\in H_0^{2,1}(\Omega_T).$$
    \end{lemma}
    \begin{proof}
        Since $\A$ is a second order differential operator, it holds for $v\in H^{2,1}_0(\Omega_T)$ that $w = v_t + \A v \in L^2(\Omega_T)$. Hence, $v$ forms a weak solution to
        \begin{equation*}
            \left\{
            \begin{array}{rll}
                v_t + \A v & = w & \text{in } \Omega_T, \\
                v & = 0 & \text{on }\Sigma_T,\\
                v(T) &= 0 & \text{in }\Omega,
            \end{array}
            \right.
        \end{equation*}
        which depends continuously on $w$ (cf.\ Section~\ref{sec:governing-dynamics}), i.e.\ there are $c_2, \tilde c_2 > 0$, such that
        \begin{equation*}
            \begin{split}
                \Vert v \Vert_{2,1,\Omega_T}^2 &\leq \tilde c_2 \Vert w\Vert^2_{0,\Omega_T}\\
                &= \tilde c_2 \left\Vert v_t + \A v \right\Vert^2_{0,\Omega_T}\\
                &\leq c_2 \left(\left\Vert v_t\right\Vert^2_{0,\Omega_T} + \Vert \A v \Vert^2_{0,\Omega_T}\right)\\
                & = c_2 \opnorm{v}^2.
            \end{split}
        \end{equation*}
        The second estimate
        $$\opnorm{v}^2 = \Vert v_t\Vert^2_{0,\Omega_T} + \Vert \A v \Vert^2_{0,\Omega_T}\leq \Vert v \Vert^2_{2,1,\Omega_T}$$
        follows directly from the definition of $\Vert \cdot \Vert_{2,1,\Omega_T}$
    \end{proof}

    \begin{theorem}\label{thm:first-existence-theorem}
        There exists a unique solution $\tilde p\in H^{2,1}_0(\Omega_T)$  to the weak formulation \eqref{eq:weak_normal}.
    \end{theorem}
    \begin{proof}
        It follows from the consecutive application of \eqref{eq:elliptic-continuous}, the Poincaré inequality and $H^{2,1}(\Omega_T)\hookrightarrow C(0,T;H^1(\Omega))$ that
        \begin{equation*}
            \begin{split}
            \int_\Omega A\nabla v(0) \cdot \nabla w(0) + \left(a_0 + \frac{1}{\alpha}\right)&v(0) w(0)\intd x\\
            &\leq  c \Vert v(0)\Vert_{1,\Omega}  \Vert w(0)\Vert_{1,\Omega}+ \frac{1}{\alpha} \Vert v(0)\Vert_{0,\Omega}  \Vert w(0)\Vert_{0,\Omega}\\
            & \leq \tilde c \Vert \nabla v(0)\Vert_{0,\Omega} \Vert \nabla w(0)\Vert_{0,\Omega}\\
            & \leq \tilde c \Vert v\Vert_{C(0,T;H^1(\Omega))} \Vert w\Vert_{C(0,T;H^1(\Omega))}\\
            &\leq \overline c \Vert v\Vert_{2,1,\Omega_T} \Vert w\Vert_{2,1,\Omega_T}
            \end{split}
        \end{equation*}
        with constants $c, \tilde c, \overline c>0$. In combination with the Cauchy–Schwarz inequality and Lemma~\ref{lem:equivalence-norms}, this yields
        \begin{equation*}
            \begin{split}
                B(v, w) & \leq \Vert v_t \Vert_{0, \Omega_T}\Vert w_t \Vert_{0, \Omega_T} + \Vert \A v\Vert_{0, \Omega_T}\Vert \A w\Vert_{0, \Omega_T} + \overline c \Vert v\Vert_{2,1,\Omega_T} \Vert w\Vert_{2,1,\Omega_T}\\
                & \leq C \Vert v\Vert_{2,1,\Omega_T} \Vert w\Vert_{2,1,\Omega_T},
            \end{split}
        \end{equation*}
        where $C>0$ is constant. Furthermore, it holds with \eqref{eq:elliptic-coercive} for some $\gamma > 0$ and $v\in H^{2,1}(\Omega_T)$ that
        \begin{equation*}
                B(v,v) \geq \int_{\Omega_T} (v_t)^2 + (\A v)^2\intd x \intd t + \gamma \Vert v(0)\Vert^2_{1,\Omega} + \frac{1}{\alpha} \int_\Omega v(0)^2\intd x \geq \opnorm{v}^2
        \end{equation*}
        Therefore, Lemma~\ref{lem:equivalence-norms} implies the existence of $C>0$, such that
        $$B(v,v) \geq C \Vert v \Vert_{2,1,\Omega_T}^2.$$
        In summary, $B$ is continuous and coercive and since $L$ is linear and bounded, the Lax-Milgram theorem implies the existence and uniqueness of a solution $\tilde p\in H_0^{2,1}(\Omega_T)$ to problem \eqref{eq:weak_normal}.
    \end{proof}

    \subsection{Mixed formulation and semi-time discretization}\label{sec:mixed-form}
    An auxiliary variable $\tilde q$ may be used in order to write \eqref{eq:homogenized-fourth-order} as a coupled system in $\tilde p$ and $\tilde q$, which allows a discretization with piecewise linear and continuous finite elements. After the definition of $\tilde q := \A \tilde p$ we may find a solution to \eqref{eq:homogenized-fourth-order} by considering the mixed formulation
    \begin{equation}\label{eq:mixed_formulation} 
        \left\{
        \begin{array}{rll}
            - \tilde p_{tt} + \A\tilde q &=\; \tilde f & \text{in }\Omega_T,\\
            - \A \tilde p + \tilde q &=\; 0 & \text{in }\Omega_T,\\
            \tilde p &=\; 0 & \text{on } \Sigma_T,\\
            \tilde q &=\; 0 &\text{on } \Sigma_T,\\
            \left(- \tilde p_t + \A\tilde p + \frac{1}{\alpha}\tilde p\right)(0) &=\; 0& \text{in } \Omega,\\
            \tilde p(T) &=\; 0 & \text{in } \Omega.
        \end{array}
        \right.
    \end{equation}
    In order to set up a weak formulation for this, we introduce the function spaces $P := \{v \in H^1(0,T;H_0^1(\Omega))\mid v(T) = 0\text{ in }\Omega\}$, $Q := L^2(0,T; H_0^1(\Omega))$ and the product space $X:= P \times Q$. Now, consider the bilinear form $B_M: X \times X \to \R$ with
    \begin{equation*}
        \begin{split}
            B_M((v_1, v_2), (w_1, w_2)) &=  \int_{\Omega_T} (v_1)_t (w_1)_t + v_2 w_2 + a_0(v_2w_1-v_1w_2)\intd x\intd t\\
            &\quad + \int_{\Omega_T}  A \nabla v_2 \cdot \nabla w_1- A\nabla v_1 \cdot \nabla w_2 \intd x\intd t\\
            &\quad + \int_\Omega A \nabla v_1(0) \cdot \nabla w_1(0) + \left(a_0 + \frac{1}{\alpha}\right)v_1(0) w_1(0)\intd x\\
        \end{split}
    \end{equation*}
    and the linear form $L_M: X \to \R$ with
    $$L_M(w_1, w_2) = \int_{\Omega_T} \tilde f w_1 \intd x \intd t.$$

    \begin{definition}[Mixed weak formulation]
        The weak formulation of \eqref{eq:mixed_formulation} is given by: let $\tilde q := \A\tilde p$, find $(\tilde p, \tilde q)\in X$, such that
        \begin{equation}\label{eq:weak_mixed}
            B_M((\tilde p, \tilde q), (w_1, w_2)) = L_M(w_1, w_2)\quad \forall (w_1,w_2)\in X.
        \end{equation}
    \end{definition}

    The proof of existence, uniqueness and equivalence to \eqref{eq:fourth-order-system} simplifies due to the results from the last section.\\

    \begin{theorem}\label{thm:equivalence-mixed}
        The weak mixed formulation \eqref{eq:weak_mixed} has a unique solution which coincides with the solution to the fourth order system \eqref{eq:homogenized-fourth-order} under appropriate regularity assumptions.
    \end{theorem}
    
    \begin{proof}
        Let $\tilde p\in H^{2,1}(\Omega_T)\cap L^2(0,T; H^3(\Omega))$ be a solution to \eqref{eq:homogenized-fourth-order} and $\tilde q= \A \tilde p\in L^2(0,T; H^1(\Omega))$. Then we have from \eqref{eq:weak_normal} that
        \begin{equation}\label{eq:mixed-proof-from-before}
            \begin{split}
                \int_{\Omega_T} \tilde f w_1\intd x\intd t &= \int_{\Omega_T} \tilde p_t (w_1)_t + \tilde q \A w_1 \intd x \intd t\\
                &\quad + \int _\Omega A \nabla \tilde p(0) \cdot \nabla w_1(0) + \left(a_0 + \frac{1}{\alpha}\right)\tilde p(0) w_1(0)\intd x
            \end{split}
        \end{equation}
        for all $w_1\in H^1(0,T; C_0^\infty(\Omega))$. It holds for arbitrary $w_2\in Q$ that
        \begin{equation*}
        \int_{\Omega_T} (\tilde q - \A \tilde p)w_2\intd x\intd t = 0.
        \end{equation*}
        Adding this to the right hand side of \eqref{eq:mixed-proof-from-before} and inserting the definition of $\A$ at both occurrences, this results in
        \begin{equation*}
            \begin{split}
            \int_{\Omega_T} \tilde f w_1\intd x\intd t & =
                \int_{\Omega_T} \tilde p_t (w_1)_t + \tilde q w_2 + a_0 \left(\tilde q w_1 - \tilde p w_2\right)\intd x\intd t\\
                &\quad + \int_{\Omega_T}\nabla \tilde q \cdot A \nabla w_1 - A \nabla \tilde p \cdot \nabla w_2 \intd x \intd t\\
                &\quad+ \int _\Omega A \nabla \tilde p(0) \cdot \nabla w_1(0) + \left(a_0 + \frac{1}{\alpha}\right)\tilde p(0) w_1(0)\intd x,
            \end{split}
        \end{equation*}
        which is equivalent to \eqref{eq:weak_mixed} by the density of $\C^\infty_0(\Omega )$ in $H^1_0(\Omega)$. Therefore, we have successfully shown the equivalence of \eqref{eq:weak_normal} and \eqref{eq:weak_mixed}. Since a solution to \eqref{eq:weak_normal} exists by Theorem~\ref{thm:first-existence-theorem}, \eqref{eq:weak_mixed} also has a solution. We continue with showing that \eqref{eq:weak_mixed} still admits only one solution in the context of less regularity than required for \eqref{eq:weak_normal}.
        
        Suppose that $(\tilde p_1,\tilde q_1),(\tilde p_2,\tilde q_2)\in X$ both solve \eqref{eq:weak_mixed}. Then $(\tilde p, \tilde q) = (\tilde p_1 - \tilde p_2, \tilde q_1 - \tilde q_2)$ satisfies
        \begin{equation*}
            \begin{split}
                0 & =
                \int_{\Omega_T} \tilde p_t (w_1)_t + \tilde q w_2 + a_0 \left(\tilde q w_1 - \tilde p w_2\right)\intd x\intd t\\
                &\quad + \int_{\Omega_T}\nabla \tilde q \cdot A \nabla w_1 - A \nabla \tilde p \cdot \nabla w_2 \intd x \intd t\\
                &\quad+ \int _\Omega A \nabla \tilde p(0) \cdot \nabla w_1(0) + \left(a_0 + \frac{1}{\alpha}\right)\tilde p(0) w_1(0)\intd x
            \end{split}
        \end{equation*}
        for all $(w_1,w_2)\in X$. Testing this equation with $w_1=\tilde p$ and $w_2=\tilde q$ and using the symmetry of $A$ yields
        \begin{equation*}
            \int_{\Omega_T} \tilde p_t^2 + \tilde q^2\intd x\intd t + \int_\Omega A \nabla \tilde p(0) \cdot \nabla \tilde p(0) + \left(a_0 + \frac{1}{\alpha}\right)\tilde p(0)^2\intd x = 0.
        \end{equation*}
        Now, it follows from $\A$ being elliptic and $\alpha >0$ that $\tilde p = 0$ and $\tilde q = 0$. Hence, a solution to \eqref{eq:weak_mixed} is unique in $X$.
    \end{proof}

    \subsection{Numerical examples}\label{sec:numerical-examples-1}
    In the numerical examples of this section, we investigate the method of solving the fourth order system instead of the data assimilation problem as presented in the previous sections. In those sections a homogenized version of \eqref{eq:fourth-order-system} has been introduced for the demonstration of the idea behind the reformulation method. For the implementation of the numerical examples, we now use a non-homogenized version of \eqref{eq:fourth-order-system}. Following the lines in Theorem~\ref{thm:equivalence-normal} and Theorem~\ref{thm:equivalence-mixed}, we use a Petrov-Galerkin approach to derive the mixed weak formulation
    \begin{equation}\label{eq:weak_non_homogenized}
    \tilde B_M((p,q), (w_1,w_2)) = \tilde L_M (w_1, w_2) \quad \forall (w_1,w_2)\in X,
    \end{equation}
    where we have $p\in P$ and, in contrast to the homogenized setting, $q\in \tilde Q := \left\{v \in L^2(0,T; H^1(\Omega)) \mid v = - y^{(d)}\right\}$. With $\tilde X := P\times \tilde Q$, the bilinear form $\tilde B_M:\tilde X\times X \to \R$ is given as
    \begin{equation*}
        \begin{split}
            \tilde B_M((v_1,v_2), (w_1,w_2)) &=  \int_{\Omega_T} (v_1)_t (w_1)_t + v_2 w_2 + a_0(w_1v_2-v_1w_2)\intd x\intd t\\
            &\quad + \int_{\Omega_T}  A \nabla w_1 \cdot \nabla v_2- A\nabla v_1 \cdot \nabla w_2 \intd x\intd t\\
            &\quad + \int_\Omega A \nabla v_1(0) \cdot \nabla w_1(0) + \left(a_0 + \frac{1}{\alpha}\right)v_1(0) w_1(0)\intd x,
        \end{split}
    \end{equation*}
    while the linear form $\tilde L_M:X\to\R$ is 
    $$\tilde L_M(w_1,w_2) = \int_{\Omega_T} \left(f - y_t^{(d)}- \A y^{(d)}\right)w_1\intd x \intd t + \int_\Omega \left(y^{(b)} - y^{(d)}(0)\right) w_1(0)\intd x.$$
    In both of the following examples, the systems have been solved on a domain with $\Omega = [0,1]$ and $T=1$. For both spatial and temporal discretization, linear finite elements have been applied on an equidistant grid with $\Delta t = \Delta x = 0.025$. All implementations have been realized in Python. The code of the presented examples is available at \url{https://doi.org/10.5281/zenodo.13133804}.\\

    \begin{example}[Compensation of an inaccurate background guess]\label{ex:inacc_background}
        We investigate the ability of the introduced method to correct an inaccurate background guess. Therefore, we assume that the reality can be described precisely by the model, which leaves the background guess $y^{(b)}(x)$ as the only error inflicted quantity.

        Consider the one dimensional homogeneous heat equation with $\A = - \nu \partial^2_x$, $\nu >0$ and $f(t,x) = 0$. We assume that the reality can be described by
        $$y^{(d)}(t,x) = \sin(\pi x) e^{-\pi^2\nu t}.$$
        While the true initial state of the dynamics is $y^{(d)}(0,x) = \sin(\pi x)$, we choose an error inflicted background guess $y^{(b)}(x)$, such that
        \begin{enumerate}[label=(\roman*)]
            \item it approximately depicts the reality, i.e.
            $y^{(b)}(x) = -\left(x - \frac{1}{2}\right)^2 + \frac{1}{4}$,
            \item it substantially differs from the reality, i.e.
            $y^{(b)}(x) = \sin(2\pi x)$.
        \end{enumerate}
        The goal is now to find an updated initial condition $y(0,x) = y_0(x)$, which approximates $y^{(d)}(0,x)$ better that $y^{(b)}$. Since $y^{(d)}(t,x)$ is an exact solution to \eqref{eq:concrete-state} for $f(t,x) = 0$ and the control $u=\sin(\pi x)$, this leads to a better approximation of $y^{(d)}(t,x)$ with $y(t,x)$ in $\Omega_T$. Since small values of $\alpha$ reduce the influence of the faulty background guess in the optimization problem, one might suspect the updated initial condition $y_0(x)$ and the model $y(t,x)$ on entire $\Omega_T$ to be closer to reality as the values of $\alpha$ get smaller. The results for the background guesses (i) and (ii) are shown in Figures~\ref{fig:inacc_background_good} and \ref{fig:inacc_background_bad}. Furthermore, Table~\ref{tab:inacc_background} complements the examples by providing error measurements for different choices of $\alpha$.\\
    \end{example}

    \begin{figure}
            \centering
            \begin{subfigure}[t]{0.32\textwidth}
                \centering
                \includegraphics[width=\linewidth]{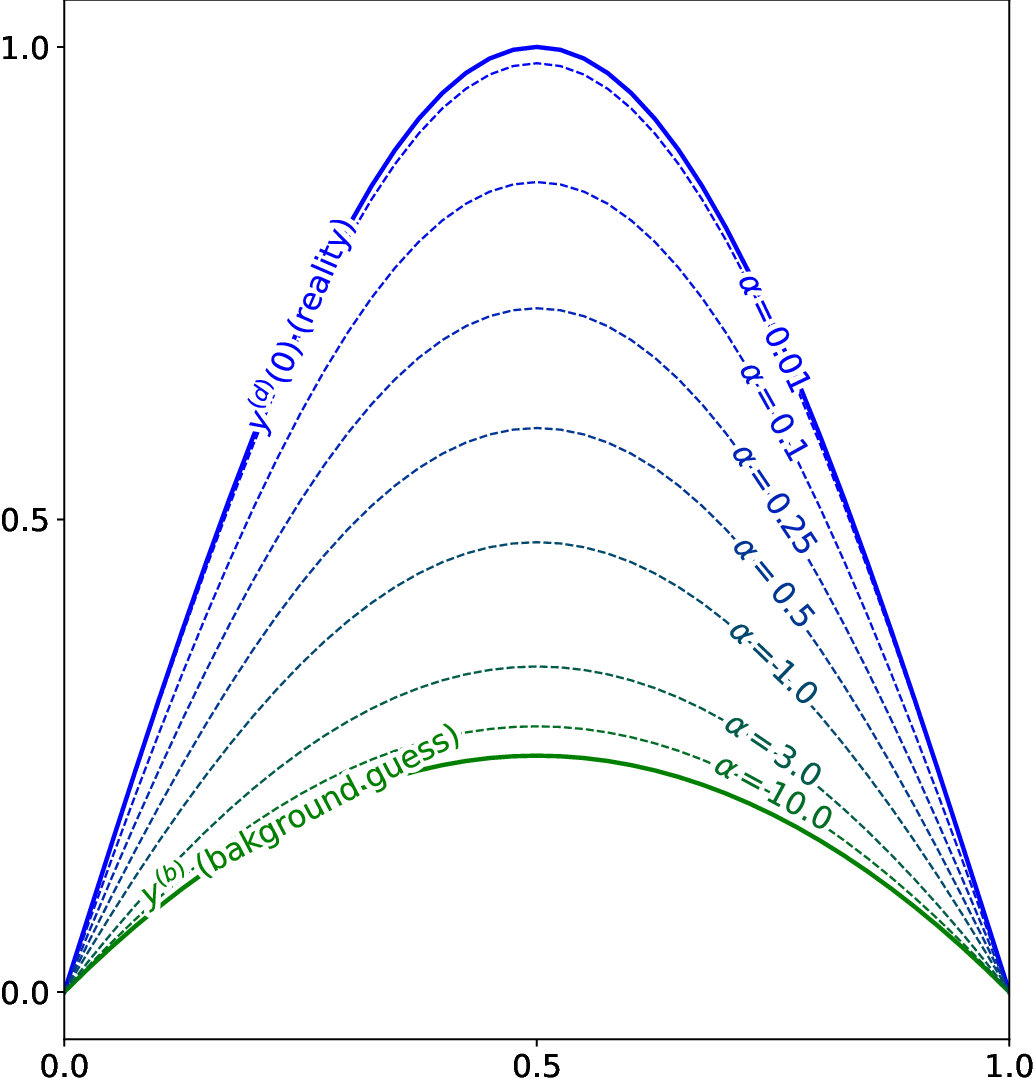}
                \caption{Reality $y^{(d)}(0,x)$ at $t=0$, background guess $y^{(b)}(x)$ and controlled initial conditions for different values of $\alpha$}
            \end{subfigure}
            \hspace*{\fill}
            \begin{subfigure}[t]{0.32\textwidth}
                \centering
                \includegraphics[width=\linewidth]{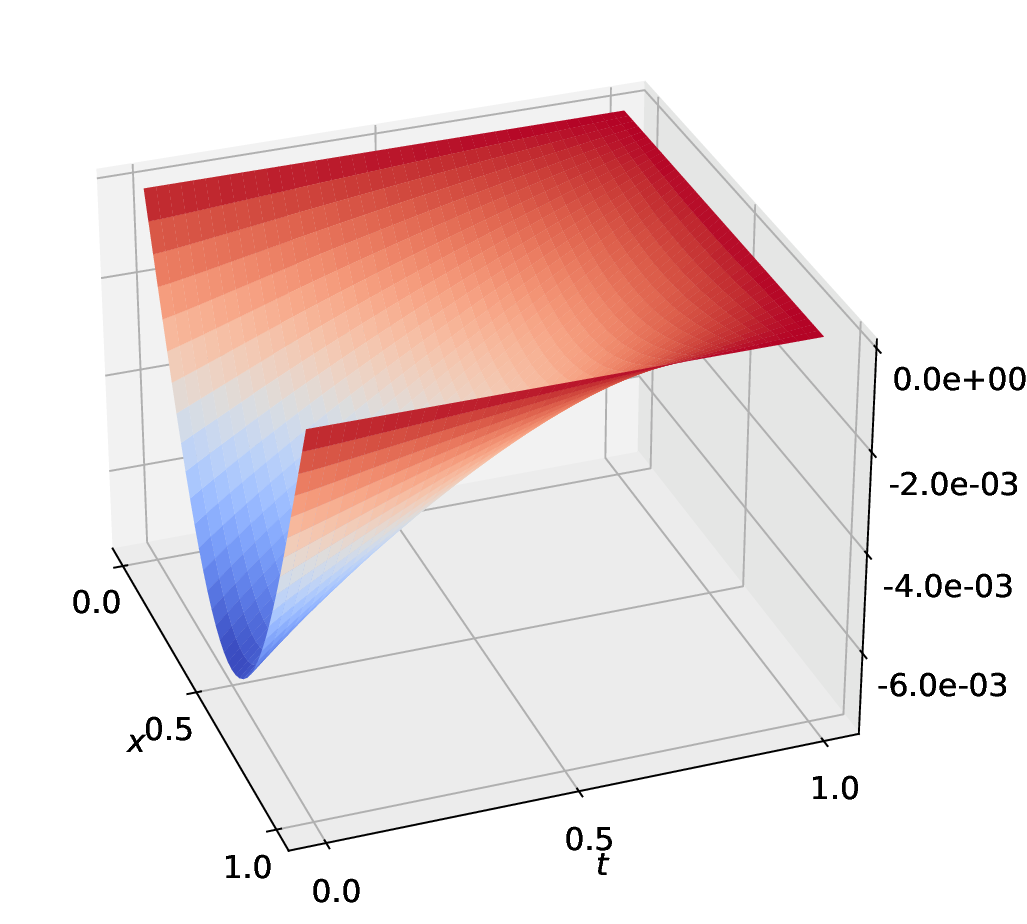}
                \caption{Adjoint state $p(t,x)$ for $\alpha = 0.01$}
            \end{subfigure}
            \begin{subfigure}[t]{0.32\textwidth}
                \centering
                \includegraphics[width=\linewidth]{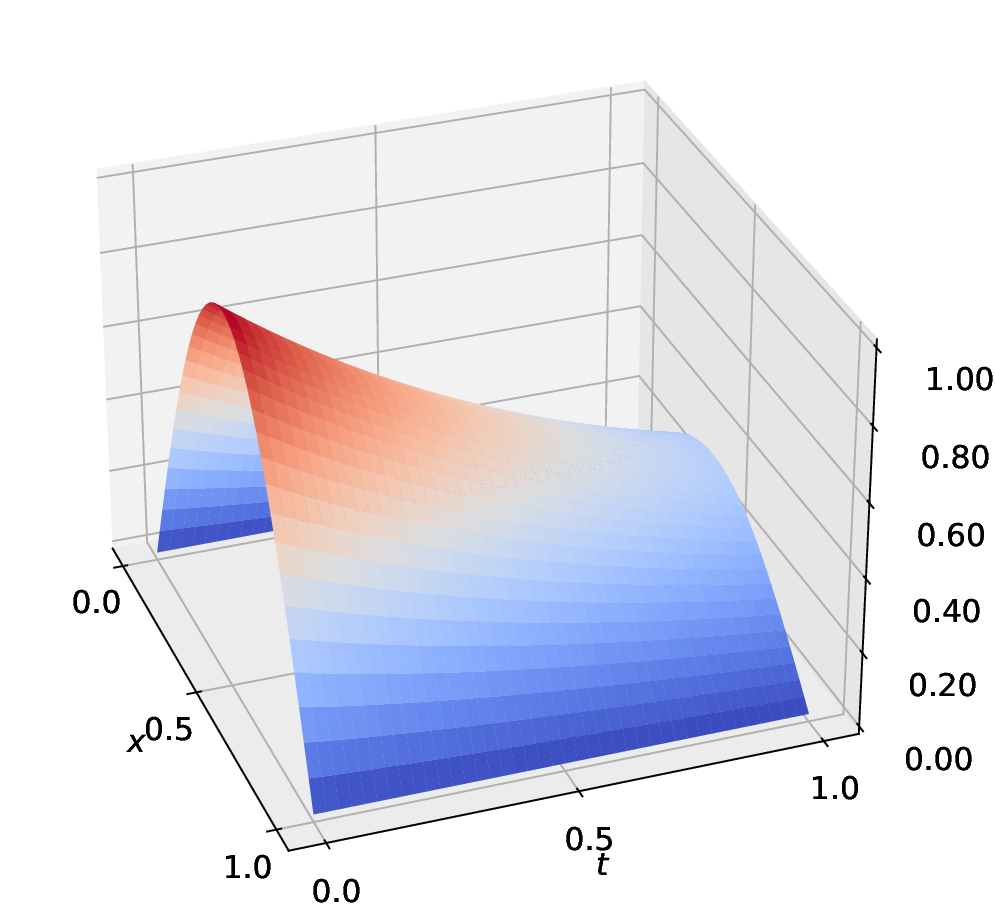}
                \caption{State $y(t,x)$, computed for controlled initial condition with $\alpha= 0.01$}
            \end{subfigure}
            \caption{Results for Example~\ref{ex:inacc_background} (i) with $y^{(b)}(x) = -\left(x - \frac{1}{2}\right)^2 + \frac{1}{4}$}
            \label{fig:inacc_background_good}
        \end{figure}

        \begin{figure}
            \centering
            \begin{subfigure}[t]{0.32\textwidth}
                \centering
                \includegraphics[width=\linewidth]{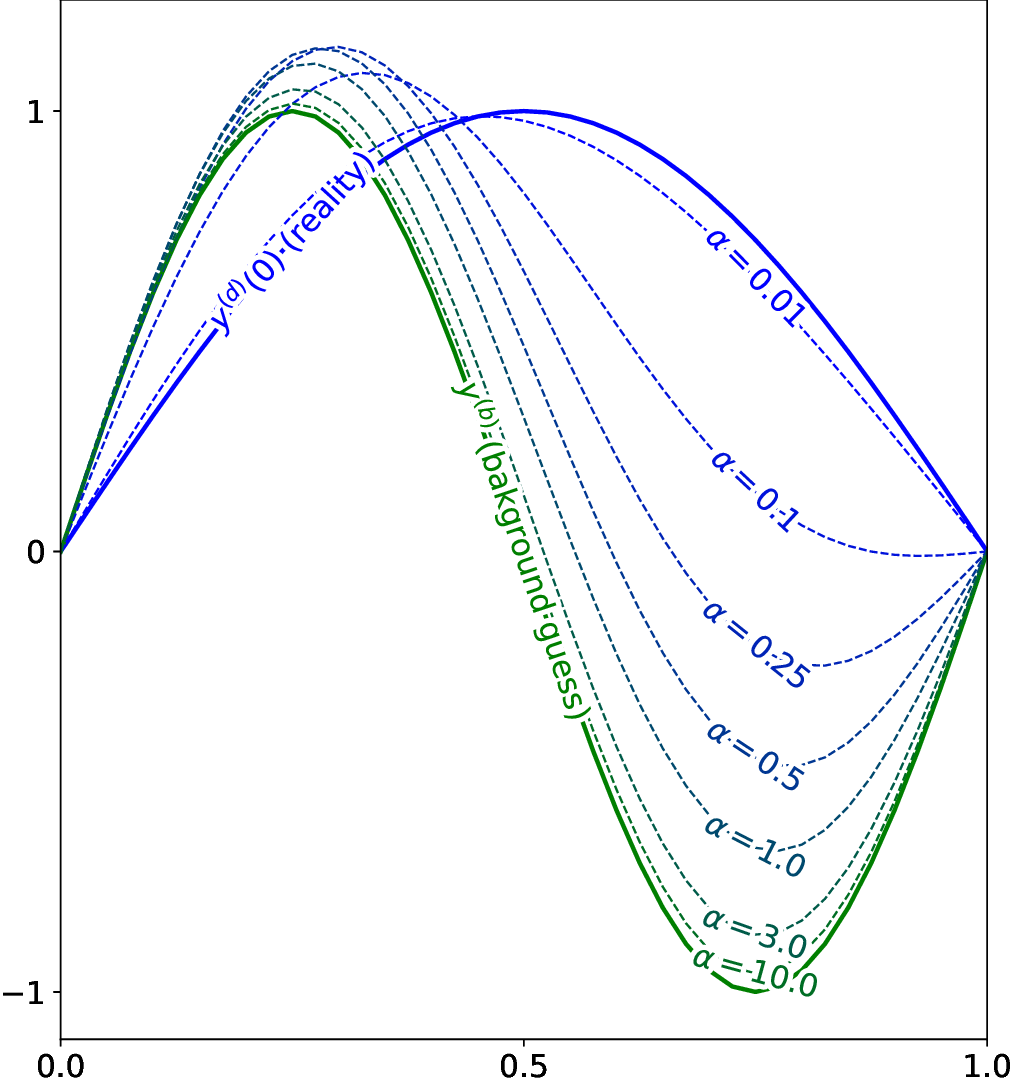}
                \caption{Reality $y^{(d)}(0,x)$ at $t=0$, background guess $y^{(b)}(x)$ and controlled initial conditions for different values of $\alpha$}
            \end{subfigure}
            \hspace*{\fill}
            \begin{subfigure}[t]{0.32\textwidth}
                \centering
                \includegraphics[width=\linewidth]{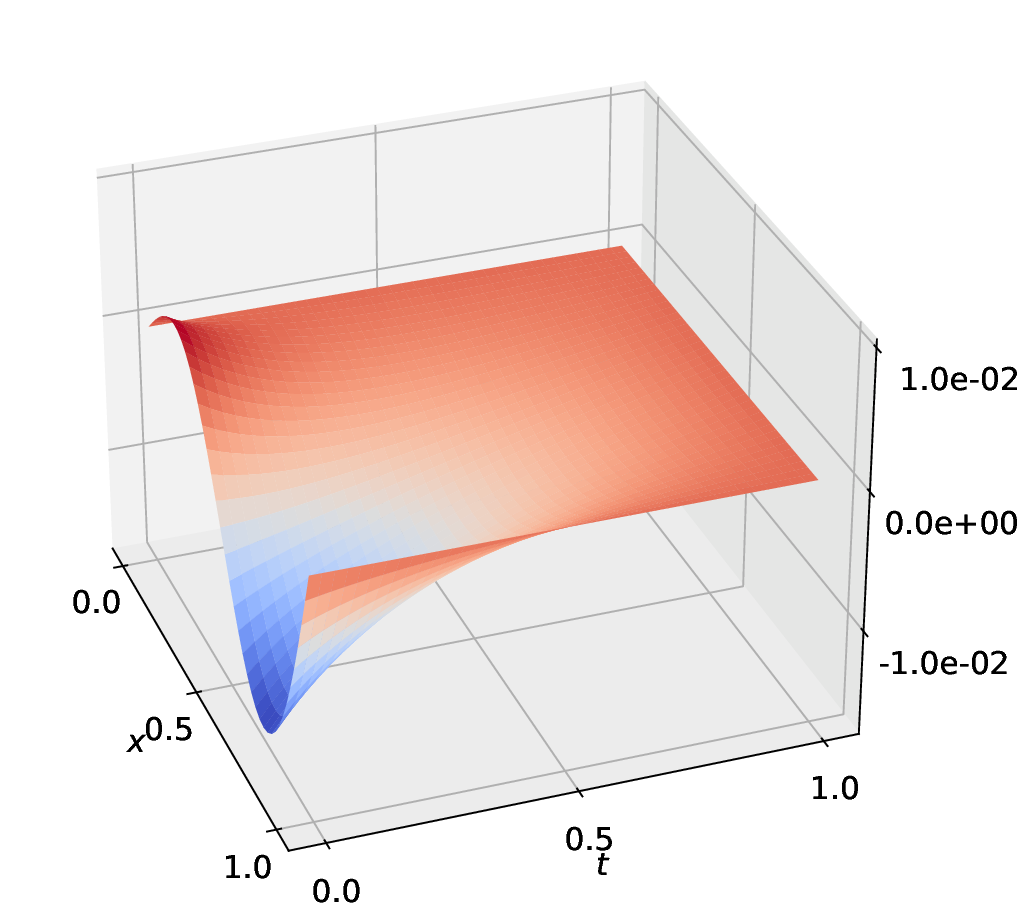}
                \caption{Adjoint state $p(t,x)$ for $\alpha = 0.01$}
            \end{subfigure}
            \begin{subfigure}[t]{0.32\textwidth}
                \centering
                \includegraphics[width=\linewidth]{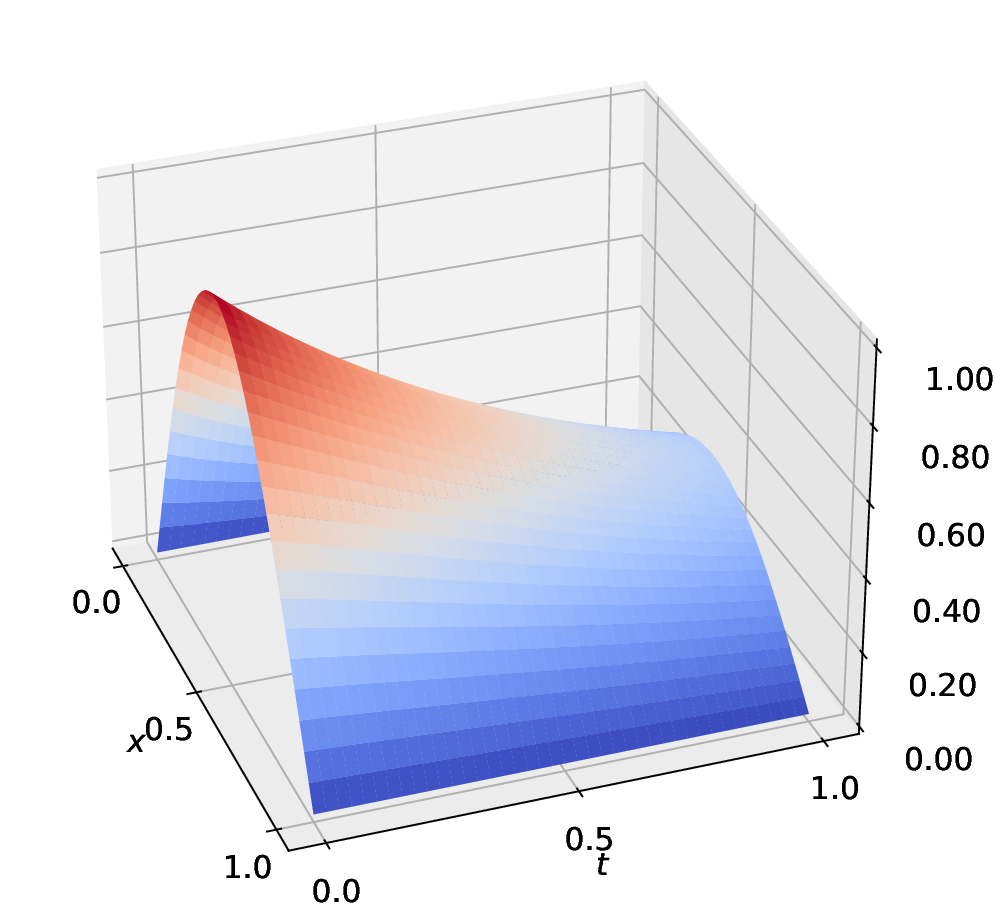}
                \caption{State $y(t,x)$, computed for controlled initial condition with $\alpha= 0.01$}
            \end{subfigure}
            \caption{Results for Example~\ref{ex:inacc_background} (ii) with $y^{(b)}(x) = \sin(2\pi x)$}
            \label{fig:inacc_background_bad}
        \end{figure}
        
        \begin{table}
            \small
            \centering
            \begin{tabular}{ccccccccc}\hline
                 & $y^{(b)}$&$\alpha=0.01$& $\alpha=0.1$ & $\alpha=0.25$ & $\alpha=0.5$ & $\alpha=1.0$& $\alpha=3.0$& $\alpha=10.0$\\\hline
                 (i)&0.3436& 0.0076 & 0.0640 & 0.1252 &0.1835&0.2392&0.3000&0.3292\\
                 (ii)& 0.5298&0.0210&0.1425&0.2401&0.3215&0.3952&0.4738&0.5114\\\hline
            \end{tabular}
            \caption{Root-mean-square error $\text{RMSE} = \sqrt{\frac{1}{Nd}\sum_{i=1}^{N}\sum_{j=1}^{d} (y^{(d)}(t_i,x_j) - y(t_i,x_j))^2}$ for controlled state $y(t,x)$ and reality $y^{(d)}(t,x)$ from Example~\ref{ex:inacc_background} with nodes $t_i$ and $x_i$ in $t$ and $x$ direction respectively with multiplicities $N$ and $d$}
            \label{tab:inacc_background}
        \end{table} 

    \begin{example}[Accounting for model errors]\label{ex:inacc_model}
        We challenge our method by assuming a good initial guess but an error inflicted model. In detail, we consider the heat equation with $\A y(t,x) = - \nu \partial^2_x y(t,x)$, $\nu >0$ and $f(t,x) = 0$ as an error inflicted model. Furthermore, we choose the background guess as 
        $$y^{(b)}(x) = \sin (\pi x),$$
        for which a solution to this equation is given as
        $$y(t,x) = \sin (\pi x) e^{-\nu \pi^2 t}.$$
        Deviating from this, we assume that the reality differs from the model and experiences an additional heat inflow, which corresponds to a model with right hand side
        $$f(t,x) = 2 \sin(\pi x)e^{-\nu \pi^2 t}\frac{1}{\pi}\frac{\varepsilon}{\varepsilon^2 + \left(t-\frac{1}{6}\right)^2}$$
        for small $\varepsilon > 0$. Hence, the heat inflow has its maximum for $t= \frac{1}{6}$ and flattens earlier and later in time. The reality with this right hand side models can be directly computed as
        \begin{equation}\label{eq:example_model_reality}
            y^{(d)}(t,x) = 2 \sin(\pi x) e^{-\nu \pi^2 t} \left(\frac{1}{\pi}\arctan \left(\frac{t-\frac{1}{6}}{\varepsilon}\right)+1\right).
        \end{equation}
        Hence, it holds that the background guess already approximates the reality at the beginning of the observations, i.e. $y^{(d)}(0,x) \approx y^{(b)}(x)$ for small $\varepsilon > 0$. The setting with which we start into the data assimilation process is depicted in Figure~\ref{fig:inacc_model_initial}. The effect after the assimilation process can be found in Figure~\ref{fig:inacc_model_after}. 
        
        While the dynamics resulting from $y^{(b)}$ as an initial condition emulates $y^{(d)}$ perfectly for $0\leq t \leq \frac{1}{6}$, it does not integrate the blowup of the reality for $t=\frac{1}{6}$ in the model. On the other hand side, the controlled model with assimilated data compensates for differing $y^{(d)}$, such that the overall error in $\Omega_T$ gets reduced.\\
    \end{example}

    \begin{figure}
            \centering
            \begin{subfigure}[t]{0.32\textwidth}
                \centering
                \includegraphics[width=\linewidth]{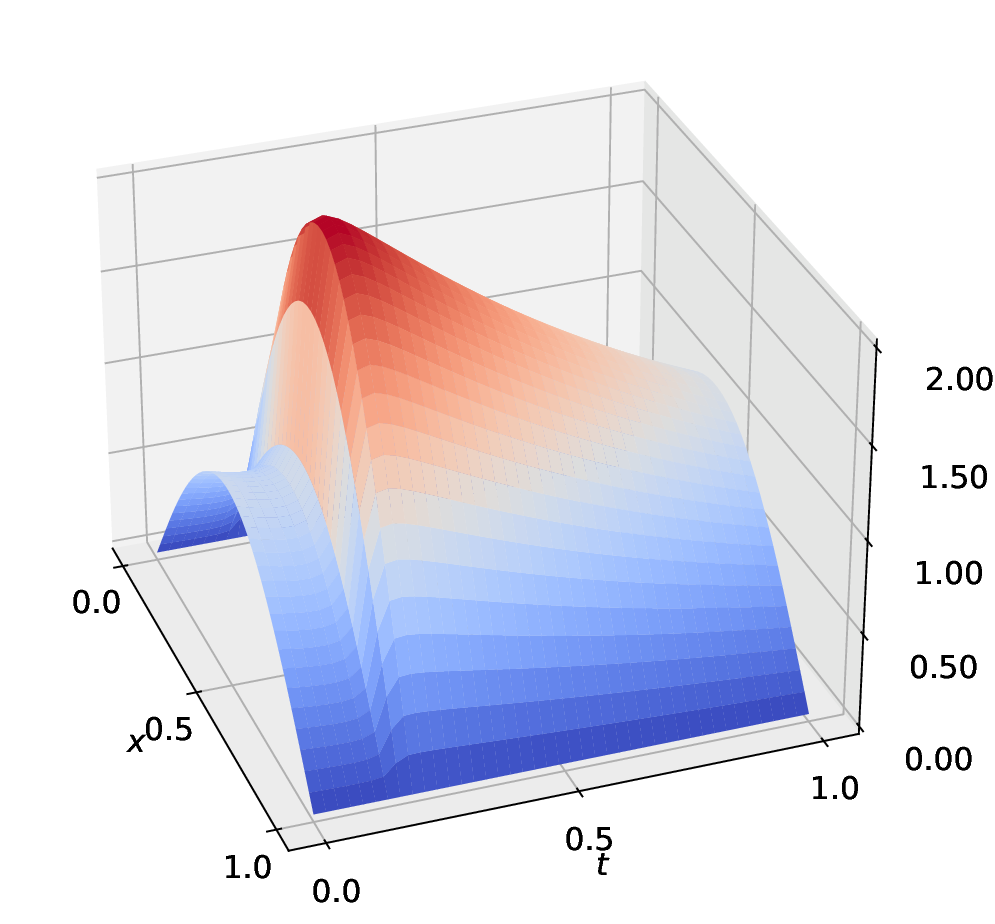}
                \caption{Reality $y^{(d)}$ as given in \eqref{eq:example_model_reality}}
            \end{subfigure}
            \hspace*{\fill}
            \begin{subfigure}[t]{0.32\textwidth}
                \centering
                \includegraphics[width=\linewidth]{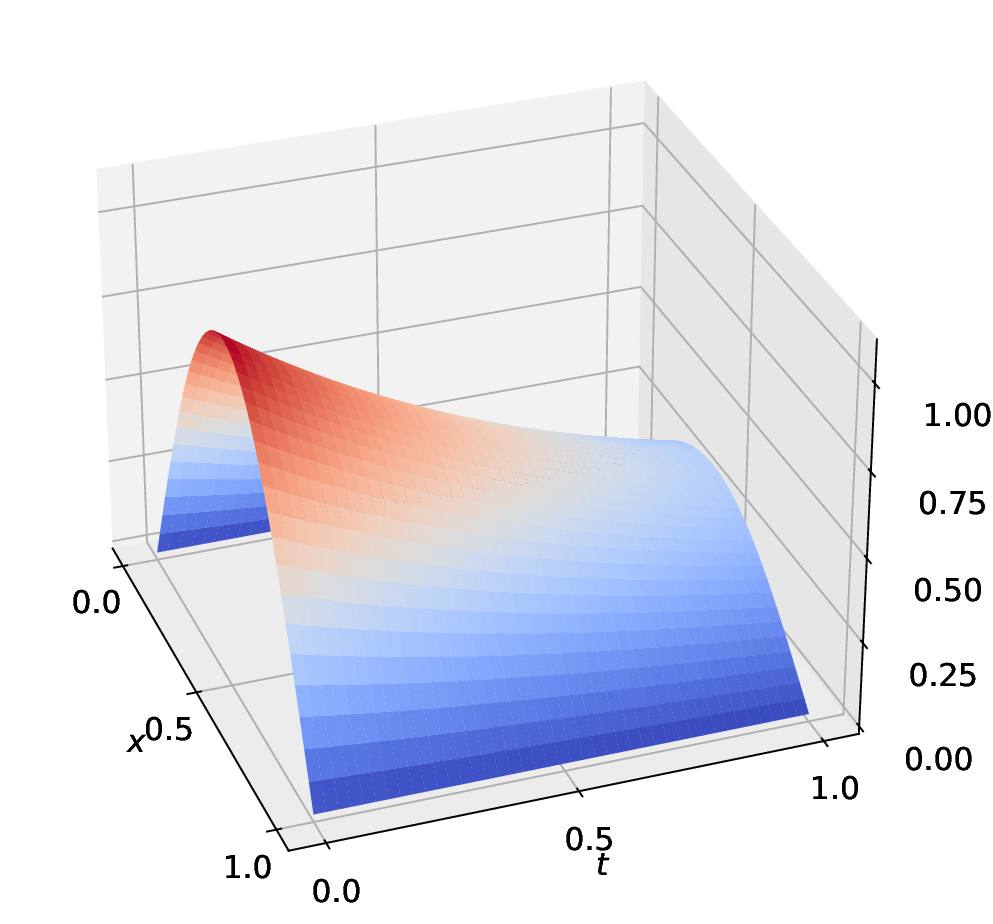}
                \caption{Computed dynamics with background guess $y^{(b)}$ as initial condition}
            \end{subfigure}
            \begin{subfigure}[t]{0.32\textwidth}
                \centering
                \includegraphics[width=0.85\linewidth]{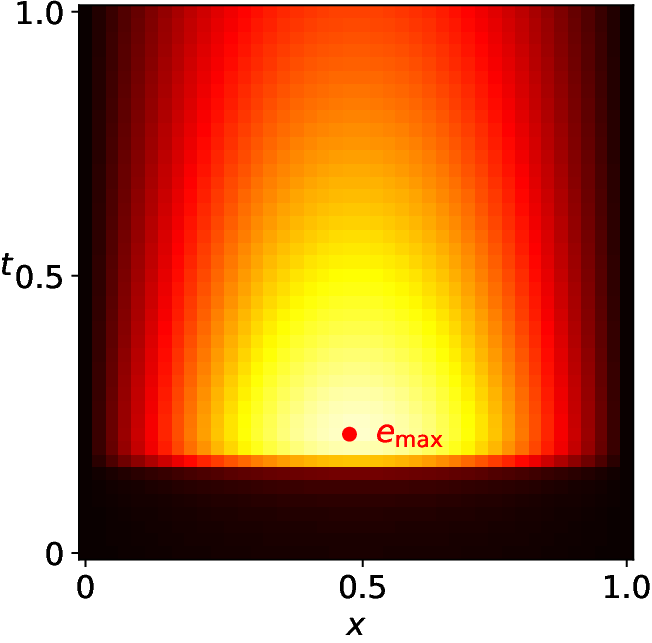}
                \caption{Absolute error between reality and dynamics with initial condition $y^{(b)}$. The maximal absolute error is $e_\text{max} \approx 1.514$}
            \end{subfigure}
            \caption{Initial situation in Example~\ref{ex:inacc_model} before the updated initial condition has been incorporated. The parameters are $\nu = 0.1$, $\alpha=0.01$ and $\varepsilon = 0.01$. The root-mean-square error (computed as for Example~\ref{ex:inacc_background} in Table~\ref{tab:inacc_background}) is $\text{RMSE} \approx 73.01 \cdot 10^{-2}$}
            \label{fig:inacc_model_initial}
        \end{figure}

        \begin{figure}
            \centering
            \begin{subfigure}[t]{0.32\textwidth}
                \centering
                \includegraphics[width=\linewidth]{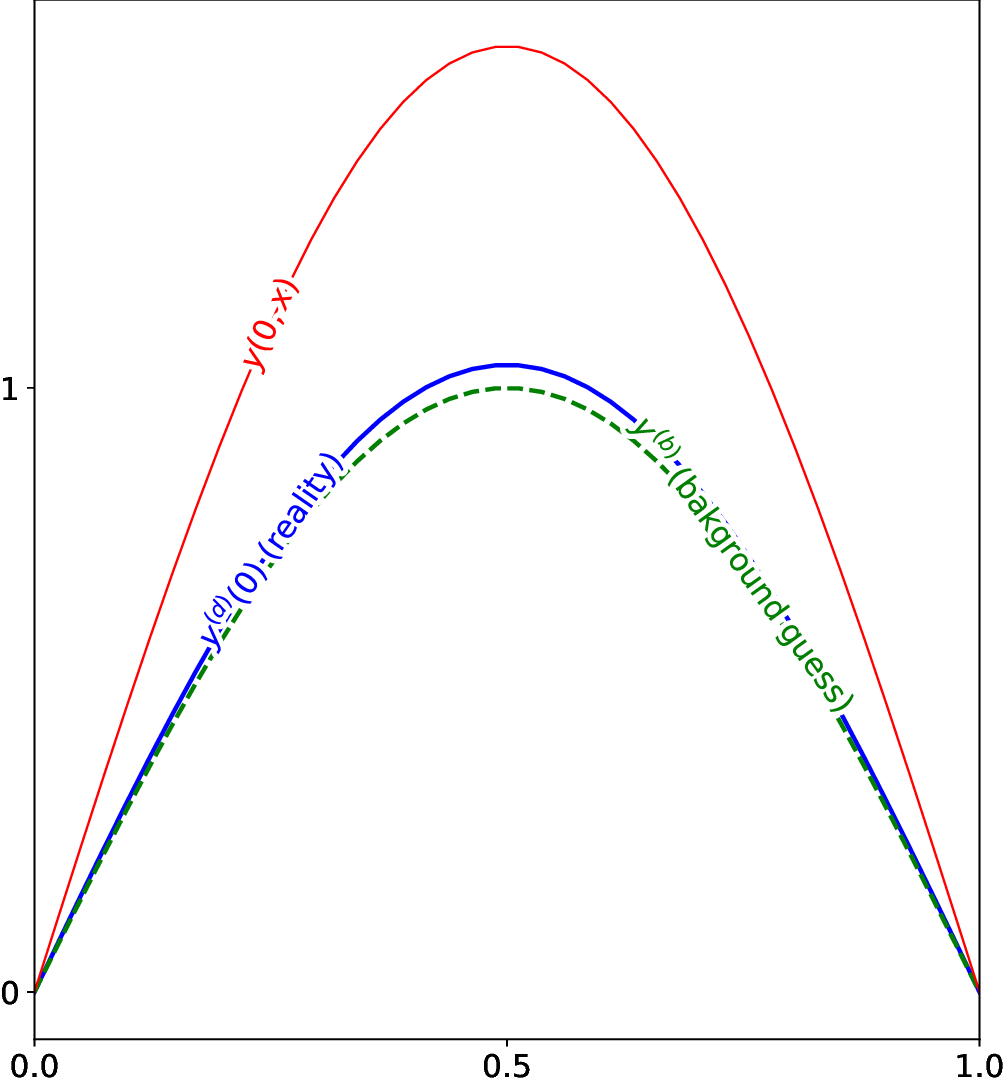}
                \caption{Reality $y^{(d)}(0,x)$ at $t=0$, background guess $y^{(b)}(x)$ and controlled initial condition}
            \end{subfigure}
            \hspace*{\fill}
            \begin{subfigure}[t]{0.32\textwidth}
                \centering
                \includegraphics[width=\linewidth]{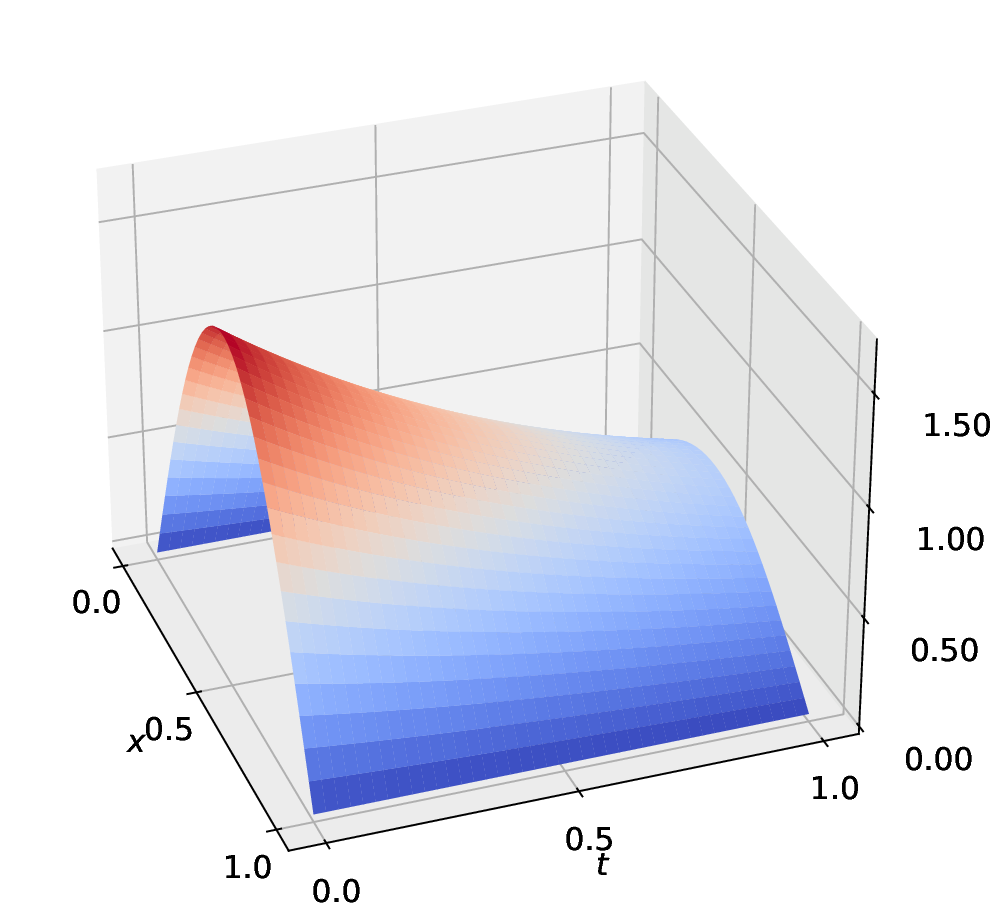}
                \caption{Computed dynamics with control as initial condition}
            \end{subfigure}
            \begin{subfigure}[t]{0.32\textwidth}
                \centering
                \includegraphics[width=0.85\linewidth]{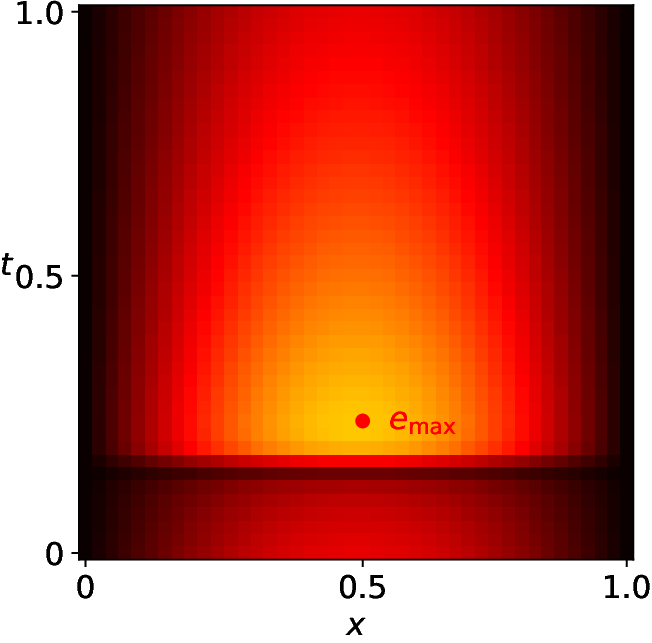}
                \caption{Absolute error between reality and controlled dynamics. The maximal absolute error is $e_\text{max} \approx 1.062$}
            \end{subfigure}
            \caption{Results from Example~\ref{ex:inacc_model}. The parameters are $\nu = 0.1$, $\alpha=0.01$ and $\varepsilon = 0.01$. The root-mean-square error (computed as for Example~\ref{ex:inacc_background} in Table~\ref{tab:inacc_background}) is $\text{RMSE} \approx 53.14 \cdot 10^{-2}$}
            \label{fig:inacc_model_after}
        \end{figure}

    \begin{remark}[Choice of $\alpha$]\label{rem:alpha}
        The presented examples and further experiments, indicated that choosing $\alpha > 0$ as small as possible supports our method best. Nevertheless, we expect this fact to change in case of error inflicted measurements $y^{(d)}$ and would like to investigate this in further detail in combination with the introduction of weighted norms in the optimal control problem \eqref{eq:control-problem}. 
    \end{remark}

    \section{Construction of an adaptive time grid}
    \label{sec:time-adaptivity}
    We now aim to improve the accuracy of $p(0)$ and in doing so the accuracy of the control $u$. For that purpose, an a-posteriori error estimate for the fourth order elliptic problem from Section~\ref{sec:elliptic_fourth_order} is derived in Section~\ref{sec:error-estimate}. For the creation of an adaptive time grid, the error estimate can be utilized in an adaptive cycle. A possible execution of this idea is presented in Section~\ref{sec:adaptive_cycle}. Afterwards, the numerical implementation of this adaptive method is discussed in Section~\ref{sec:adaptive_numerical_examples}. 
    
    The presented ideas are motivated by \cite{time-adaptivity-mpc, space-time}, where similar adaptive procedures have been accomplished for fourth order systems, which resulted from distributed control.

    \subsection{A-posteriori error estimate}
    \label{sec:error-estimate}
    In order to come up with an error estimate for the discretization error of $\tilde p(0)$, we derive an a-posteriori error bound for $(\tilde p,\tilde q)\in P\times Q$ first. The steps up to Theorem~\ref{thm:error-estimate} closely follow the ideas in \cite[Section 3.3]{time-adaptivity-mpc} and adapt them to the setting of initial control. The transition to the error estimation of $\tilde p(0)$ will be accomplished in Corollary~\ref{cor:initial-error}.

    For the measurement of the error for $(\tilde p,\tilde q)$, consider the norm
    $$\opnorm{(\tilde p,\tilde q)}_M := \left(\int_{\Omega_T} (\tilde p_t)^2 + \tilde q^2\intd x \intd t\right)^\frac{1}{2}\quad \forall (\tilde p,\tilde q)\in P\times Q,$$
    which is $\opnorm{\cdot}$ as defined in \eqref{eq:opnorm}, adjusted to the notion of the mixed formulation. 
    
    We consider for the estimation of the error a semi-time discretization of \eqref{eq:weak_mixed} with respect to $\tilde p$, while $\tilde q$ is kept continuous. Let the time horizon be subdivided into the $ N+1 $ discrete time steps $0 = \tau_0 < \tau_1 < ... < \tau_N = T$ with step sizes $\Delta \tau_i = \tau_i - \tau_{i-1}$ and denote the time intervals by $I_i = (\tau_{i-1}, \tau_i]$ for $i=1,...,N$. Moreover, consider the set of linear polynomials $\mathbb P_1$ as well as the time discrete space 
    $$V^k = \Big\{v\in C^0(0,T; H^1(\Omega))\;\Big|\; v|_{I_i}\in \mathbb P_1(I_i))\Big\}$$
     in order to define $P^k := V^k \cap P$.\\

    \begin{definition}[Semi-time discrete mixed weak form]\label{def:semi-form}
    The semi-time discrete mixed weak formulation of \eqref{eq:mixed_formulation} is given by: find $(\tilde p^k, \tilde q^k)\in P^k \times Q$ such that
    \begin{equation}\label{eq:semi-mixed-formulation}
        B_M((\tilde p^k, \tilde q^k), (w_1, w_2)) = L_M(w_1, w_2),\quad \forall (w_1, w_2)\in P^k \times Q.
    \end{equation}
    \end{definition}
    Existence and uniqueness of a solution to \eqref{eq:semi-mixed-formulation} follow with the same arguments as used in Theorem~\ref{thm:equivalence-mixed} in order to show the existence and uniqueness of \eqref{eq:weak_mixed}.

    For the derivation of the residual based error estimate for the semi-time discrete mixed form, we associate $(\tilde p^k, \tilde q^k)\in P^k\times Q$ with the residuals $R_1^k\in P^*$ and $R_2^k\in Q^*$, where
    \begin{equation*}
        \begin{split}
            R_1^k(w_1) &= \int_{\Omega_T} \tilde f w_1 - (\tilde p^k)_t(w_1)_t - A \nabla \tilde q^k \cdot \nabla w_1 - a_0 w_1 \tilde q^k \intd x \intd t\\
            &\quad - \int_\Omega A\nabla \tilde p^k(0) \cdot \nabla w_1(0) + \left(a_0 + \frac{1}{\alpha}\right)p^k(0) w_1(0) \intd x,\\
            R_2^k(w_2) &= \int_{\Omega_T} -\tilde q^k w_2 + A \nabla \tilde p^k \cdot \nabla w_2 + a_0 \tilde p^k w_2 \intd x \intd t.
        \end{split}
    \end{equation*}
    We use integration by parts and the divergence theorem to find the $L^2$-representations of $R_1^k$ and $R_2^k$ as
    \begin{equation*}
        \begin{split}
            R_1^k(w_1) &= \sum_{i=1}^N \int_{I_i}\int_\Omega \left\{\tilde f + (\tilde p^k)_{tt} - \A \tilde q^k\right\}w_1 \intd x \intd t + \sum_{i=1}^m \int_\Omega (\tilde p^k)_tw_1 \Big|_{I_i} \intd x\\
            &\quad - \int_\Omega \A \tilde p^k(0) w_1(0) + \frac{1}{\alpha} \tilde p^k(0)w_1(0)\intd x,\\
            R_2^k(w_2) &= \sum_{i=1}^N
            \int_{I_i}\int_\Omega \left\{- \tilde q^k + \A \tilde p^k\right\}w_2 \intd x \intd t.
        \end{split}
    \end{equation*}
    The residual $R_1^k$ fulfills the Galerkin orthogonality
    \begin{equation}\label{eq:galerkin-first}
        R_1^k(w_1) = 0, \quad \forall w_1 \in P^k,
    \end{equation}
    as well as
    \begin{equation}\label{eq:galerkin-second}
        R_2^k (w_2) = 0, \quad \forall w_2 \in Q.
    \end{equation}
    Furthermore, it holds for $(\tilde p, \tilde q)\in P\times Q$, $(\tilde p^k, \tilde q^k)\in P^k\times Q$ and $(w_1,w_2)\in P^k\times Q$ that
    \begin{equation}
        \begin{split}
            B_M((\tilde p - \tilde p^k, \tilde q - \tilde q^k), (w_1, w_2)) &= B_M((\tilde p, \tilde q), (w_1, w_2))-B_M((\tilde p^k, \tilde q^k), (w_1, w_2))\\
            &= L_M(w_1, w_2) - L_M(w_1, w_2) = 0.
        \end{split}
    \end{equation}
    For the relaxed choice of $(w_1, w_2)\in P\times Q$ the residual equation
    \begin{equation}\label{eq:residual-eq}
        B_M((\tilde p - \tilde p^k, \tilde q - \tilde q^k), (w_1, w_2)) = R_1^k(w_1) + R_2^k(w_2) = R_1^k(w_1)
    \end{equation}
    holds true, where the last step followed by the insertion of \eqref{eq:galerkin-second}. The a-posteriori error estimate of this section can now be formulated.\\

    \begin{theorem}\label{thm:error-estimate}
        Let $(\tilde p, \tilde q)\in P\times Q$ denote the solution to \eqref{eq:weak_mixed} and $(\tilde p^k, \tilde q^k)\in P^k\times Q$ the solution to \eqref{eq:semi-mixed-formulation}. Then, the residual based a-posteriori error estimate
        \begin{equation}\label{eq:error-estimate}
            \opnorm{(\tilde p - \tilde p^k, \tilde q - \tilde q^k)}_M^2 \leq C \eta^2
        \end{equation}
        holds true, where $C > 0$ is constant and
        \begin{equation}\label{eq:error-rhs}
            \eta^2 = \sum_{i=1}^N \int_{I_i}\int_\Omega (\Delta\tau_i)^2 \left(\tilde f + (\tilde p^k)_{tt} - \A \tilde q^k \right)^2\intd x \intd t.
        \end{equation}
    \end{theorem}
    \begin{proof}
        Let $e := (\tilde p - \tilde p^k, \tilde q - \tilde q^k)$, $w = (w_1,w_2)$ and denote the approximation of $w_1\in P$ in $P^k$ by $I^k_Pw_1 \in P^k$, such that $w_1(\tau_i) = I^k_Pw_1(\tau_i)$ holds for all $i=1,...,N$. We combine \eqref{eq:galerkin-first} with \eqref{eq:residual-eq} to find that
        \begin{equation*}
            \begin{split}
                B_M(e, w) &= R_1^k(w_1 - I^k_Pw_1)\\
                &= \sum_{i=1}^N \int_{I_i}\int_\Omega r^k(w_1 - I^k_P w_1)\intd x \intd t - \int_\Omega \left(\A \tilde p^k + \frac{1}{\alpha} \tilde p^k \right) (w_1 - I^k_Pw_1)(0)\intd x\\
                & \quad + \sum_{i=1}^N \int_\Omega (\tilde p^k)_t(w_1 - I_P^kw_1)\Big|_{I_i}\intd x,
            \end{split}
        \end{equation*}
        with $r^k := \tilde f + (\tilde p^k)_{tt} - \A \tilde q^k$. Notice that only the integrals over $I_i \times \Omega$ remain while all integrals over $\Omega$ vanish since $(w_1 - I^k_Pw_1)(\tau_i) = 0$ for $i=0,...,N$. We subsequently estimate with the Cauchy-Schwarz inequality and use standard interpolation properties (see i.e. \cite[Theorem 1.7]{interpolation-result}) to get
        \begin{equation*}
            \begin{split}
                |B_M(e, w)| &\leq \int_\Omega \left(\sum_{i=1}^N \Vert r^k \Vert_{L^2(I_i)} \Vert w_1 - I_P^k w_1 \Vert_{L^2(I_i)}\right)\intd x\\
                &\leq \int_\Omega \left(\sum_{i=1}^N \Vert r^k \Vert_{L^2(I_i)} c_1 \Delta\tau_i |w_1|_{H^1(\tilde I_i)}\right)\intd x,
            \end{split}
        \end{equation*}
        where $\tilde I_i$ denotes the set of intervals which share a vertex with $I_i$ and $|\cdot|_{H^1}$ the $H^1$-seminorm. With the Cauchy-Schwarz inequality for sums, it holds for constants $c_1,c_2 >0$ that
        \begin{equation*}
            \begin{split}
                |B_M(e, w)| &\leq c_1 \int_\Omega \left(\sum_{i=1}^N \Vert r^k \Vert^2_{L^2(I_i)} (\Delta\tau_i)^2\right)^\frac{1}{2} \left(\sum_{i=1}^N |w_1|_{H^1(\tilde I_i)}\right)^\frac{1}{2}\intd x\\
                &\leq c_2 \int_\Omega \left(\sum_{i=1}^N \Vert r^k \Vert^2_{L^2(I_i)} (\Delta\tau_i)^2\right)^\frac{1}{2} |w_1|_{H^1(0,T)} \intd x\\
                &\leq c_2 \left(\int_\Omega \sum_{i=1}^N \Vert r^k \Vert^2_{L^2(I_i)} (\Delta\tau_i)^2 \intd x\right)^\frac{1}{2} \left(\int_\Omega |w_1|^2_{H^1(0,T)} \intd x\right)^\frac{1}{2},\\
            \end{split}
        \end{equation*}
        where the last step followed with Hölder's inequality. The latter factor in the previous inequality can be estimated by
        $$\left(\int_\Omega |w_1|^2_{H^1(\tilde I_i)} \intd x\right)^\frac{1}{2} \leq \left(\int_{\Omega_T} (w_1)_t^2 + w_2^2 \intd x \intd t\right)^\frac{1}{2}= \opnorm{w}_M.$$
        Therefore, it holds that
        $$|B_M(e,e)| \leq c_2 \left(\int_\Omega \sum_{i=1}^N \Vert r^k \Vert^2_{L^2(I_i)} (\Delta\tau_i)^2 \intd x\right)^\frac{1}{2}\cdot \opnorm{e}_M$$
        and since $\opnorm{e}_M^2 \leq |B_M(e,e)|$ follows from the definition of $\opnorm{\cdot}_M$, we have
        $$\opnorm{e}_M^2 \leq c_2 \left(\int_\Omega \sum_{i=1}^N \Vert r^k \Vert^2_{L^2(I_i)} (\Delta\tau_i)^2 \intd x\right)^\frac{1}{2}\cdot \opnorm{e}_M,$$
        which yields the claim after squaring both hand sides.
    \end{proof}

    With Theorem \ref{thm:error-estimate}, we have a tool to estimate the error between the semi-time discrete and the continuous adjoint state $p$ in $\Omega_T$. Nevertheless, since the only information we need from this system for the calculation of the updated initial condition is the adjoint state at time $t=0$, we are rather interested in an estimation of $\Vert \tilde p(0) - \tilde p^k(0) \Vert_{L^2(\Omega)}$. In fact, this is a direct consequence of Theorem \ref{thm:error-estimate}.
    
    Denote $r = \tilde p - \tilde p^k$. The fact that $\tilde p(T) = \tilde p^k(T) = 0$ and the Cauchy-Schwarz inequality imply
    \begin{equation*}
        \begin{split}
            \Vert \tilde p(0) - \tilde p^k(0) \Vert^2_{L^2(\Omega)} &= \int_\Omega \left(r(T) - r(0)\right)^2\intd x\\
            &= \int_\Omega \left(\int_0^T r_t\cdot 1 \intd t\right)^2 \intd x\\
            &\leq T \int_{\Omega_T} r_t^2 \intd x \intd t\\
            & \leq T \int_{\Omega_T} r_t^2 + \left(\tilde q - \tilde q^k\right)^2 \intd x \intd t\\
            &= T \opnorm{(\tilde p - \tilde p^k, \tilde q - \tilde q^k)}_M^2.
        \end{split}
    \end{equation*}
    Hence, Theorem~\ref{thm:error-estimate} can be extended as follows.\\
    \begin{corollary}\label{cor:initial-error}
        Let $(\tilde p, \tilde q)\in P\times Q$ denote the solution to \eqref{eq:weak_mixed} and $(\tilde p^k, \tilde q^k)\in P^k\times Q$ the solution to \eqref{eq:semi-mixed-formulation}. Then the residual based a-posteriori error estimate
        \begin{equation}\label{eq:error-corollary-estimate}
            \Vert\tilde p(0) - \tilde p^k(0)\Vert^2_{L^2(\Omega)} \leq \tilde C \eta^2
        \end{equation}
        holds true, where $\tilde C > 0$ is constant and
        \begin{equation}\label{eq:error-corollary-rhs}
            \eta^2 = \sum_{i=1}^N \int_{I_i}\int_\Omega (\Delta\tau_i)^2 \left(\tilde f + (\tilde p^k)_{tt} - \A \tilde q^k \right)^2\intd x \intd t.
        \end{equation}
    \end{corollary}

    \begin{remark}
        Let $(p,q)\in P\times \tilde Q$ denote a solution to the non-homogenized problem \eqref{eq:weak_non_homogenized} and $(p^k, q^k)\in P^k\times \tilde Q$ the solution to a semi-time discrete mixed weak form (defined similarly to the homogenized version in Definition~\ref{def:semi-form}). Then, there exist constants $C,\tilde C > 0$, such that 
        \begin{equation}\label{eq:error-non-hom}
            \opnorm{( p - p^k, q - q^k)}_M^2 \leq C \eta^2 \quad\text{and}\quad  \Vert p(0) -  p^k(0)\Vert^2_{L^2(\Omega)} \leq \tilde  C \eta^2
        \end{equation}
        where $\eta$ is in both cases defined as
        \begin{equation}\eta^2 = \sum_{i=1}^N \int_{I_i}\int_\Omega (\Delta\tau_i)^2 \left( f - y_t^{(d)} - \A y^{(d)}+ ( p^k)_{tt} - \A q^k \right)^2\intd x \intd t.
        \end{equation}
    \end{remark}

    \subsection{Adaptive cycle}
    \label{sec:adaptive_cycle}
    The error estimate from Corollary \ref{cor:initial-error} can now be used in order to create an adaptive grid in time on which the fourth order system can be solved. We propose doing this by following the standard repetitive adaptive cycle
    \begin{equation*}
        \text{solve }\rightarrow\text{ estimate }\rightarrow\text{ mark }\rightarrow\text{ refine}.
    \end{equation*}
    Each new cycle starts with solving the fourth order system with finite elements in space and time on a coarse grid. Afterwards, one may use \eqref{eq:error-corollary-rhs} to determine the contribution of the error on each time interval $I_i$. Then, the intervals with the largest error are marked using the Dörfler marking strategy \cite{doerfler}. During the refinement step, the marked intervals are then split by e.g.\ bisection. These steps may then be repeated until a prescribed precision or a maximal number of time intervals is reached. \\
    
    \begin{remark}[Heuristic assumption]\label{rem:heuristic_assumption}
        We derived the error estimate \eqref{eq:error-estimate} (and subsequently \eqref{eq:error-corollary-estimate}) for a time discrete formulation in $\tilde p$, while $\tilde q$ is kept continuous. In practice, we solve a fully space-time discrete formulation, but still use the semi-time discrete form to construct an adaptive time grid. For this, we assume that the temporal discretization of $\tilde p^k$ is insensitive with respect to the spatial discretization. Moreover, we assume that the temporal discretization of $\tilde q^k$ does not strongly influence the error estimate. The assumption is supported by the results of our numerical examples, but may not hold true in general. For this reason, we intend to derive an a-posteriori error estimate for a fully space-time discrete mixed variational form in future work. Notice, that this assumption was also required in the works which motivated the time adaptivity in this paper \cite{time-adaptivity-mpc, alla18}.\\
    \end{remark}

    \begin{remark}[Comparison of error estimate and actual error]
        A comparison of the actual error with the estimated error bound shows that the error estimate can only be used as an indication for the change in the true error due to its different magnitude in comparison to the true error. This behavior has also be observed for Example \ref{ex:error_estimation} in Section \ref{sec:adaptive_numerical_examples} (i.e.\ Figure~\ref{fig:jump-adaptive-overview}~(c)).
    \end{remark}
    
    \subsection{Numerical examples}
    \label{sec:adaptive_numerical_examples}
    We extend the examples from Section~\ref{sec:numerical-examples-1} by applying time adaptivity using the error estimate \eqref{eq:error-non-hom}. All implementations have been realized in Python. The code of the presented examples is available at \url{https://doi.org/10.5281/zenodo.13133804}.\\

    \begin{remark}[A-priori generatable error grid]
        Since we use piecewise linear finite elements, the second order terms in the error bound vanish, such that
        $$\eta^2 = \sum_{i=1}^N \int_{I_i}\int_\Omega (\Delta\tau_i)^2 \left( f - y_t^{(d)} - \A y^{(d)}\right)^2\intd x \intd t$$
        can be used as an error indicator. Hence, the error indicator is independent of $p$ and $q$ and the time adaptive grid can be created before any computation of the adjoint state. \\
    \end{remark}
    
    \begin{example}[Continuation of Example~\ref{ex:inacc_model}]\label{ex:error_estimation}
        In this example, we extend Example~\ref{ex:inacc_model} from Section~\ref{sec:numerical-examples-1} by the application of the error estimate. The reader may refer to Example~\ref{ex:inacc_model} to get familiar with the setting. The initial situation before data assimilation on a is depicted in Figure~\ref{fig:inacc_model_initial}. The effect after the assimilation process can be found in Figure~\ref{fig:inacc_model_after}. Since we compute an adaptive time grid for the adjoint state $p(t,x)$, but still solve for $y(t)$ on a regular grid, these figures are still relatively accurate.

        The created time adaptive grid, the adjoined state on this grid and the error in the initial condition during the refinement process are depicted in Figure~\ref{fig:jump-adaptive-overview}. Furthermore, an overview of the computed initial states of the adjoint state at different refinement levels of the grid can be found in Figure~\ref{fig:jump-adaptive-initials}. In both figures, the computation of the adjoint state on an adaptive grid has been compared to its computation on a uniform grid with the same number of timesteps.
        
        We first notice, that the error bound $\eta$ with which we estimated the initial state strongly differs from the actual error in the initial state. Hence, the error bound can only be seen as an indicator for where to split the time grid and does not provide any knowledge about the magnitude of the true error.
        
        Secondly, it can be observed in this example, that the adaptive grid and the uniform grid do not differ that much if the number of timesteps is large. Nevertheless, for only a small number of timesteps, the adaptive grid is clearly advantageous over a uniform grid.\\
    \end{example}

    \begin{figure}
        \centering
        \begin{subfigure}[t]{0.3\textwidth}
            \centering
            \includegraphics[width=\linewidth]{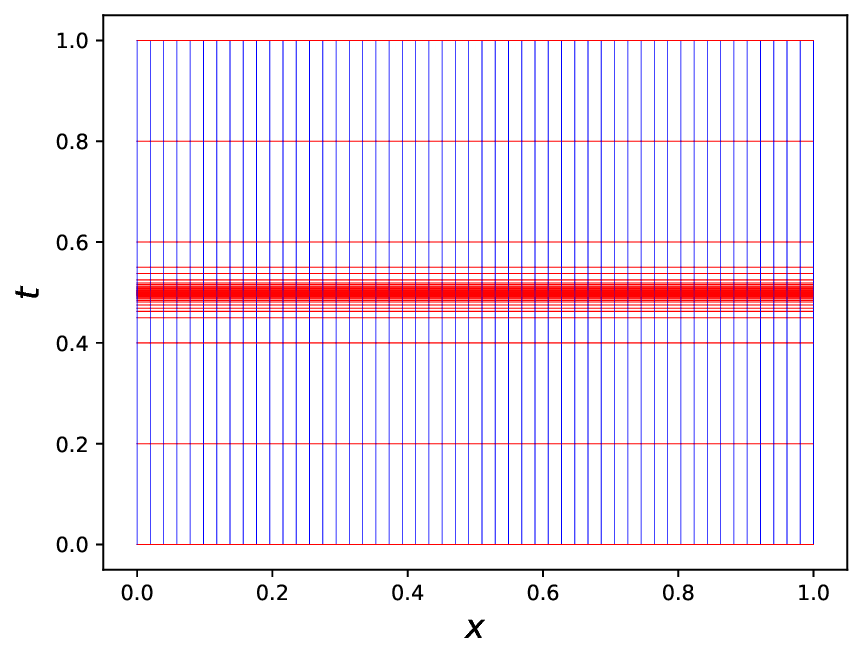}
            \caption{Time adaptive grid with $N=40$ timesteps}
        \end{subfigure}
        \hspace*{\fill}
        \begin{subfigure}[t]{0.3\textwidth}
            \centering
            \includegraphics[width=\linewidth]{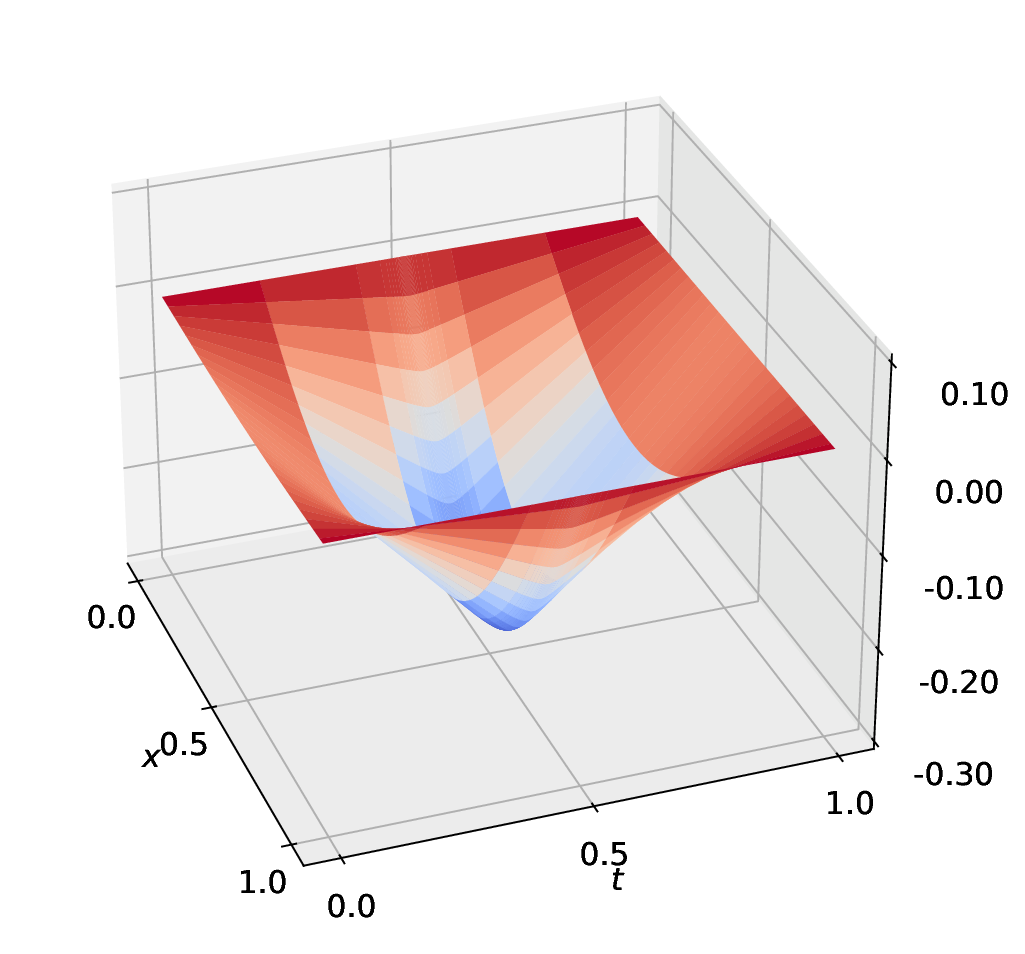}
            \caption{$p(t,x)$ on the time adaptive grid depicted from (a)}
        \end{subfigure}
        \hspace*{\fill}
        \begin{subfigure}[t]{0.3\textwidth}
            \centering
            \includegraphics[width=\linewidth]{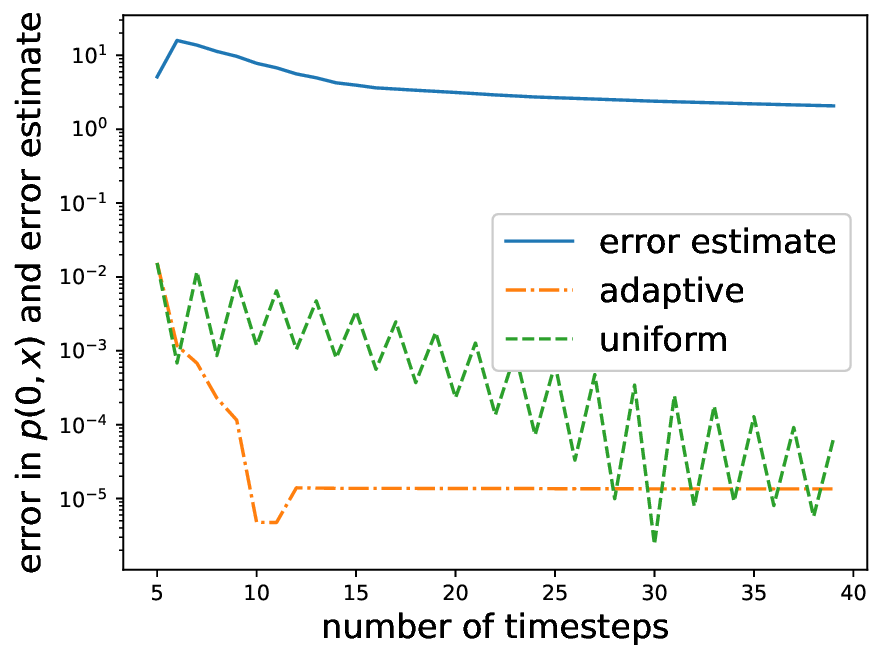}
            \caption{Comparison between the error estimate $\eta$ and the true error in the initial state $\Vert p(0) - p^k(0)\Vert_{L^2(\Omega)}$ during the creation of the time grid from (a) and on a uniform grid}
        \end{subfigure}\\
        \caption{Creation of an adaptive time grid for the computation of the adjoint state. The considered data assimilation problem is described in Example~\ref{ex:inacc_model} and Example~\ref{ex:error_estimation} with parameters $\nu = 0.1$, $\alpha=0.01$ and $\varepsilon = 0.01$. The grid was created from a grid with $N=5$ evenly spaced timesteps by bisection of the interval with the largest error estimate. The adaptive cycle has been repeated $35$ times until $N=40$ timesteps have been reached. For the comparison of $p^k(0)$ with $p(0)$, the adjoint state has been computed on a fine grid and then considered as $p(t,x)$}
        \label{fig:jump-adaptive-overview}
    \end{figure}

    \begin{figure}
        \centering
        \begin{subfigure}[t]{0.3\textwidth}
            \centering
            \includegraphics[width=\linewidth]{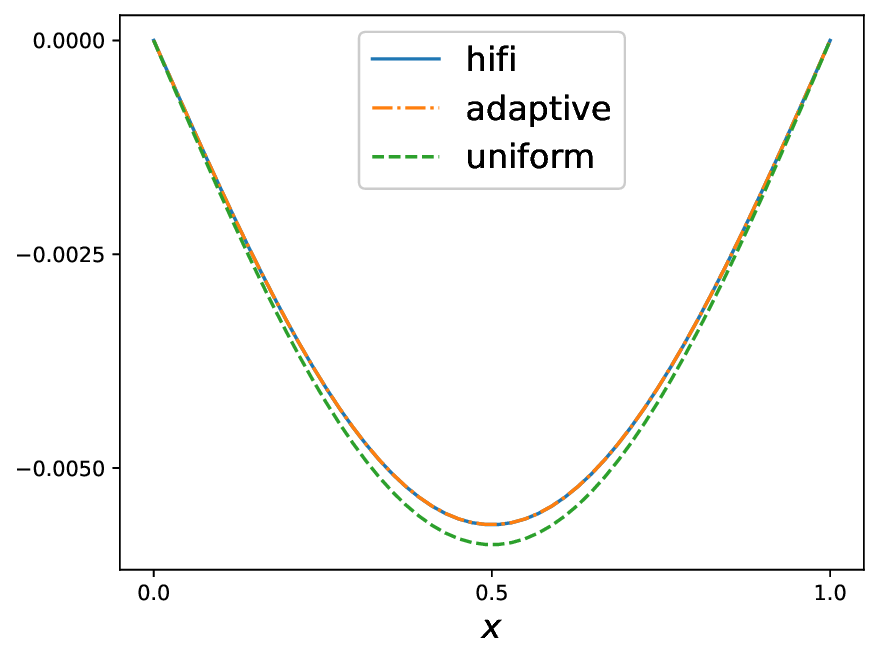}
            \caption{$N=10$}
        \end{subfigure}
        \hspace*{\fill}
        \begin{subfigure}[t]{0.3\textwidth}
            \centering
            \includegraphics[width=\linewidth]{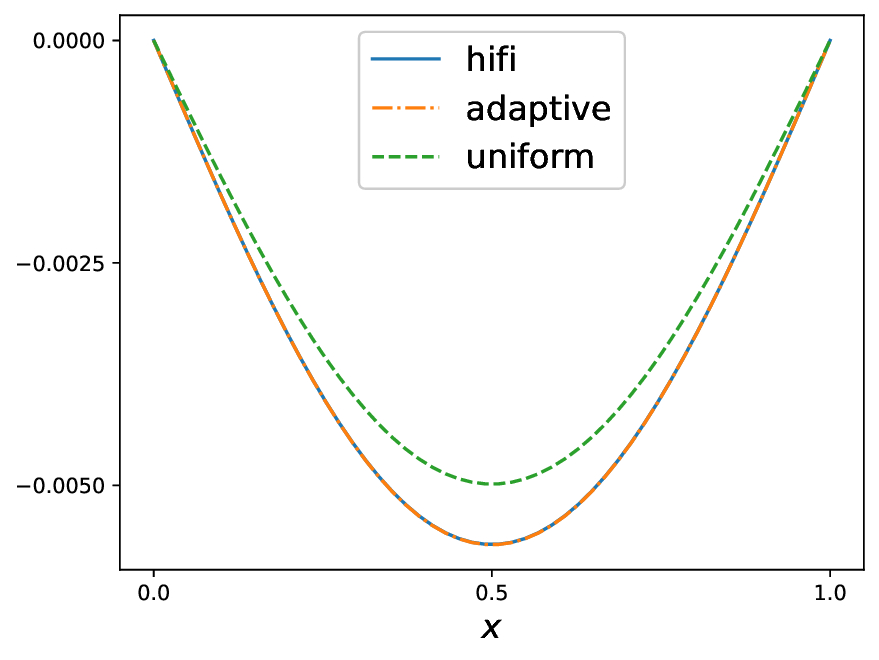}
            \caption{$N=15$}
        \end{subfigure}
        \hspace*{\fill}
        \begin{subfigure}[t]{0.3\textwidth}
            \centering
            \includegraphics[width=\linewidth]{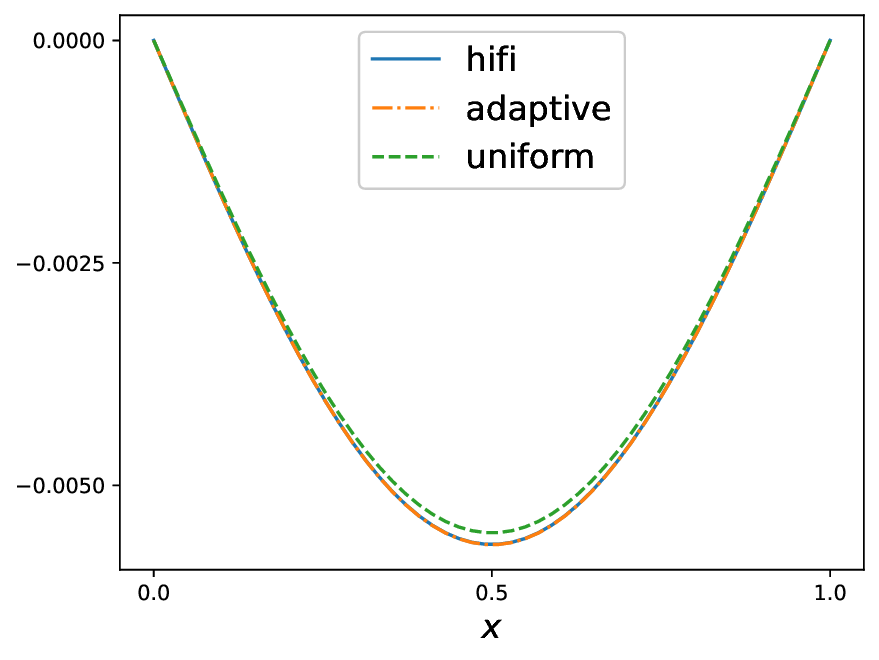}
            \caption{$N=25$}
        \end{subfigure}\\
        \caption{Comparison of initial state, computed on an adaptive grid and a uniform grid, with $p(0,x)$. The solved data assimilation problem is described in Example~\ref{ex:inacc_model} and Example~\ref{ex:error_estimation} with parameters $\nu = 0.1$, $\alpha=0.01$ and $\varepsilon = 0.01$. The grid was created from a grid with $N=5$ evenly spaced timesteps by bisection of the interval with the largest error estimate until a grid with $N=10,15,25$ was created. For the comparison of $p^k(0)$ to $p(0)$, the adjoint state has been computed on a fine grid and then considered as $p(t,x)$}
        \label{fig:jump-adaptive-initials}
    \end{figure}

    \begin{example}[Behavior of time adaptive grid]\label{ex:eps-setup}
        We investigate the response of the adaptive grid to temporal changes in the adjoint state. Therefore, we construct an adjoint state as 
        $$p(t,x) := \frac{g\left(\frac{1}{2}-\frac{t-m}{\varepsilon}\right)}{g\left(\frac{1}{2}+\frac{t-m}{\varepsilon}\right) + g\left(\frac{1}{2}-\frac{t-m}{\varepsilon}\right)} \sin (\pi x),$$
        where
        $$g(t) := \begin{cases}
            e^{-\frac{1}{t}} & t > 0,\\
            0 & \text{else}
        \end{cases}$$
        with parameters $m\in (0,T)$ and $\varepsilon >0$ such that $0 \leq m-\varepsilon< m+\varepsilon\leq T$. The parameters can be used to control the position and intensity of the adjoint's decrease in time (cf.\ Figure~\ref{fig:setup_eps}). The smaller $\varepsilon$ is, the steeper the decrease in the adjoint state over time gets. We now choose 
        \begin{equation*}
            \begin{split}
                f(t,x) &= -p_{tt}(t,x) + \nu^2\partial_x^4 p(t,x) + y^{(d)}_t(t,x) - \nu\partial_x^2 y^{(d)}(t,x),\\
                y^{(d)} (t,x) &= p_t(t,x),\\
                y^{(b)}(x) &= \left(\frac{1}{\alpha} + \nu \pi^2\right)p(0),
            \end{split}
        \end{equation*}
         such that $p(t,x)$ forms a solution to \eqref{eq:fourth-order-system} with $\A = -\nu \partial^2_{x}$. $\nu >0$. A comparison of the errors which occur in the computed initial conditions $p(0,x)$ on adaptive and uniform grids can be found in Figure~\ref{fig:eps-error}. The adaptive grids which have been produced with the adaptive cycle are depicted in Figure~\ref{fig:eps-grid}.

        For relatively large choices of $\varepsilon$, the difference between the error resulting from a uniform grid and an adaptive grid are not significant (Figure~\ref{fig:eps-error}~(d)-(f)). The adaptive grid starts showing its superiority for small choices of $\varepsilon$ (Figure~\ref{fig:eps-error}~(a)-(c)). 
          
        This observation can be explained by the examination of the corresponding grids. For large $\varepsilon$, the change in the adjoint state and therefore the time intervals in the adaptive grid are rather distributed over the entire domain, than concentrated around a single value (Figure~\ref{fig:eps-grid}~(f)). Hence, the adaptive grid and the uniform grid do not differ that much.
        The opposite is the case for small $\varepsilon$. In this case, the adaptive grid recognizes the steep change and adapts to it (Figure~\ref{fig:eps-grid}~(a)). 
    \end{example}

    \begin{figure}
            \centering
            \begin{subfigure}[t]{0.55\textwidth}
                \centering
                \includegraphics[width=\linewidth]{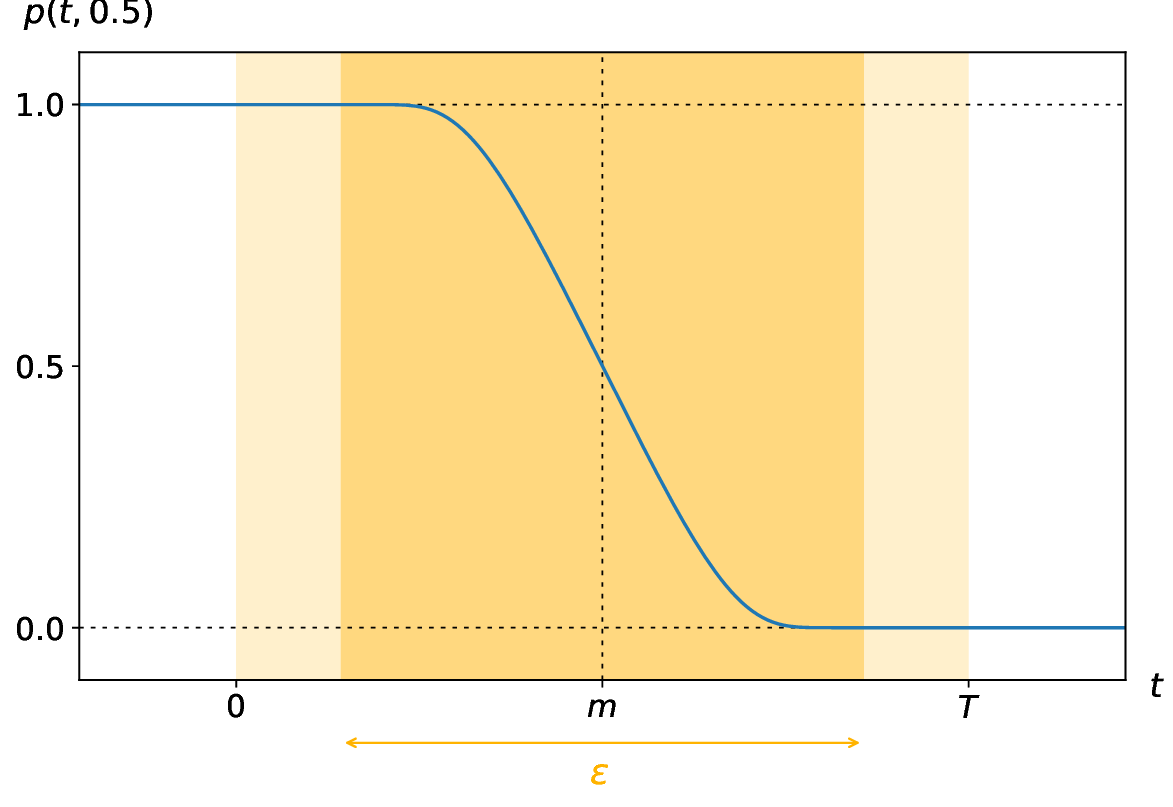}
                \caption{Cut through the adjoint state $p(t,x)$ for arguments $(t,x)\in (0,T)\times \{0.5\}$}
            \end{subfigure}
            \hspace*{\fill}
            \begin{subfigure}[t]{0.4\textwidth}
                \centering
                \includegraphics[width=\linewidth]{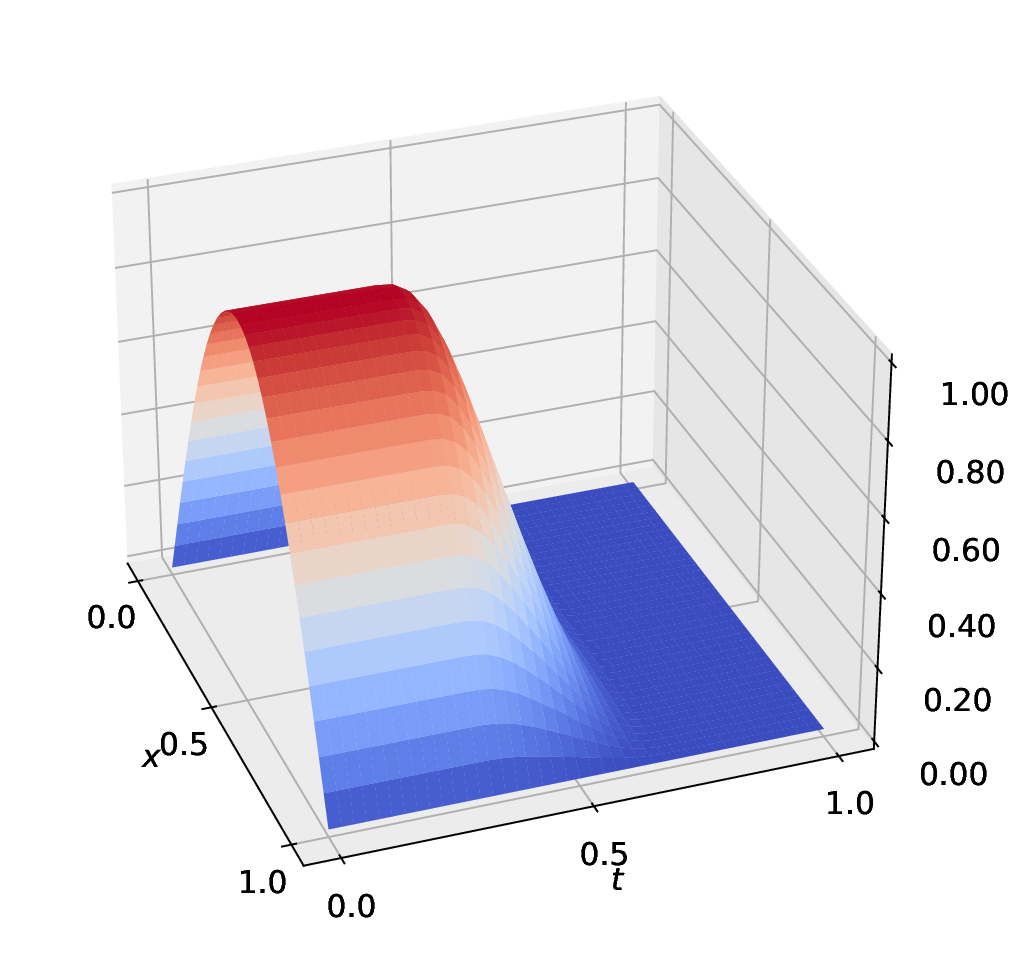}
                \caption{Adjoint state $p(t,x)$ on $\Omega_T = (0,1)^2$ for parameters $m=\frac{1}{2}$ and $\epsilon = \frac{1}{2}$}
            \end{subfigure}
            \caption{Visualization of the constructed adjoint state $p(t,x)$ in Example~\ref{ex:eps-setup}}
            \label{fig:setup_eps}
        \end{figure}

    \begin{figure}
        \centering
        \begin{subfigure}[t]{0.3\textwidth}
            \centering
            \includegraphics[width=\linewidth]{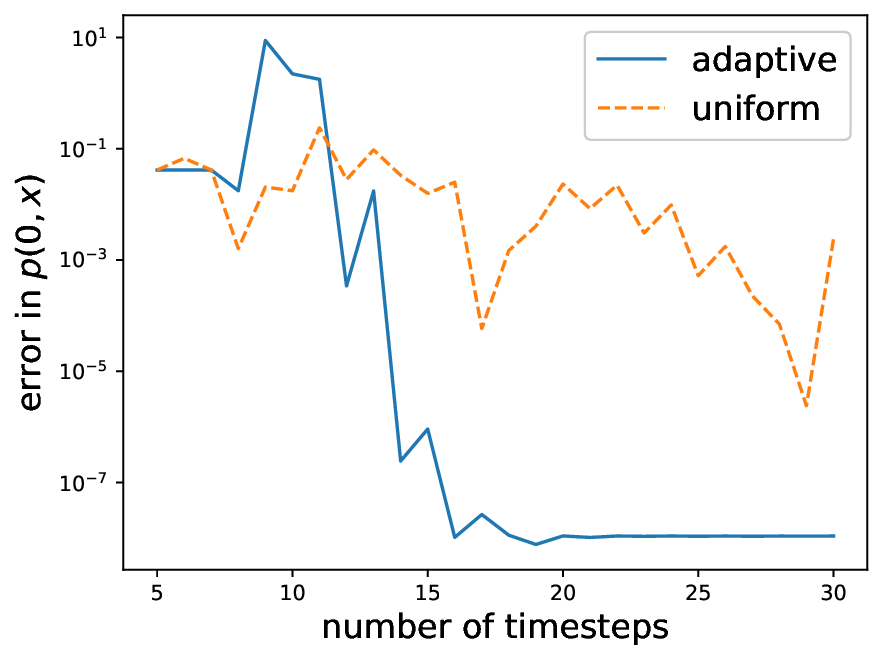}
            \caption{$\varepsilon = 0.05$}
        \end{subfigure}
        \hspace*{\fill}
        \begin{subfigure}[t]{0.3\textwidth}
            \centering
            \includegraphics[width=\linewidth]{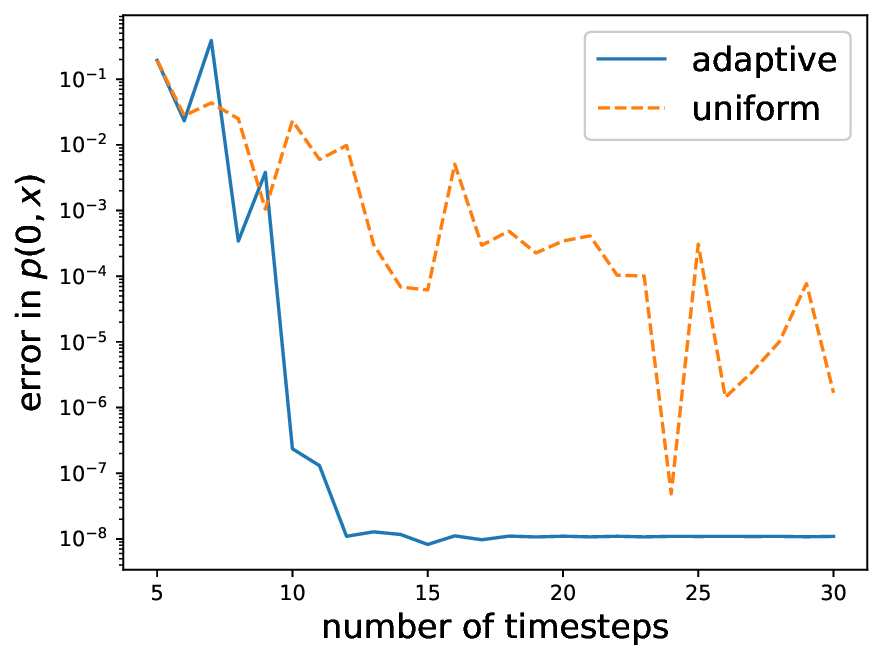}
            \caption{$\varepsilon = 0.1$}
        \end{subfigure}
        \hspace*{\fill}
        \begin{subfigure}[t]{0.3\textwidth}
            \centering
            \includegraphics[width=\linewidth]{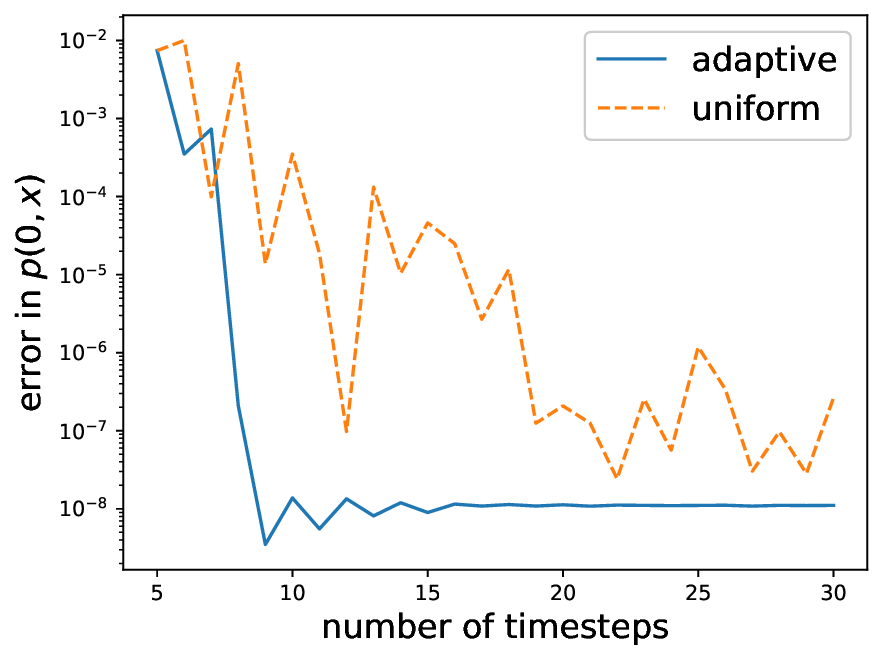}
            \caption{$\varepsilon = 0.2$}
        \end{subfigure}\\

        \centering
        \begin{subfigure}[t]{0.3\textwidth}
            \centering
            \includegraphics[width=\linewidth]{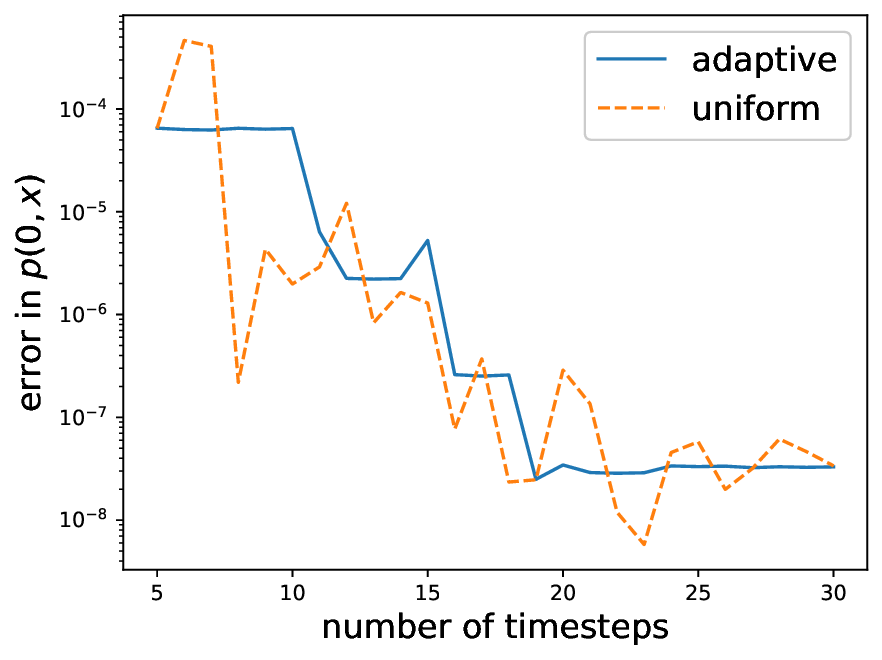}
            \caption{$\varepsilon = 0.3$}
        \end{subfigure}
        \hspace*{\fill}
        \begin{subfigure}[t]{0.3\textwidth}
            \centering
            \includegraphics[width=\linewidth]{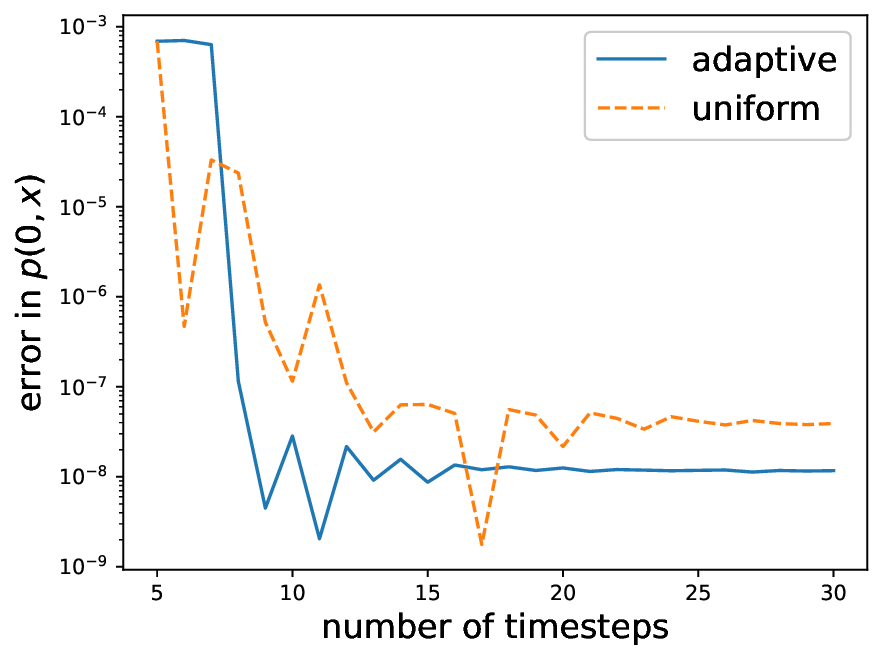}
            \caption{$\varepsilon = 0.4$}
        \end{subfigure}
        \hspace*{\fill}
        \begin{subfigure}[t]{0.3\textwidth}
            \centering
            \includegraphics[width=\linewidth]{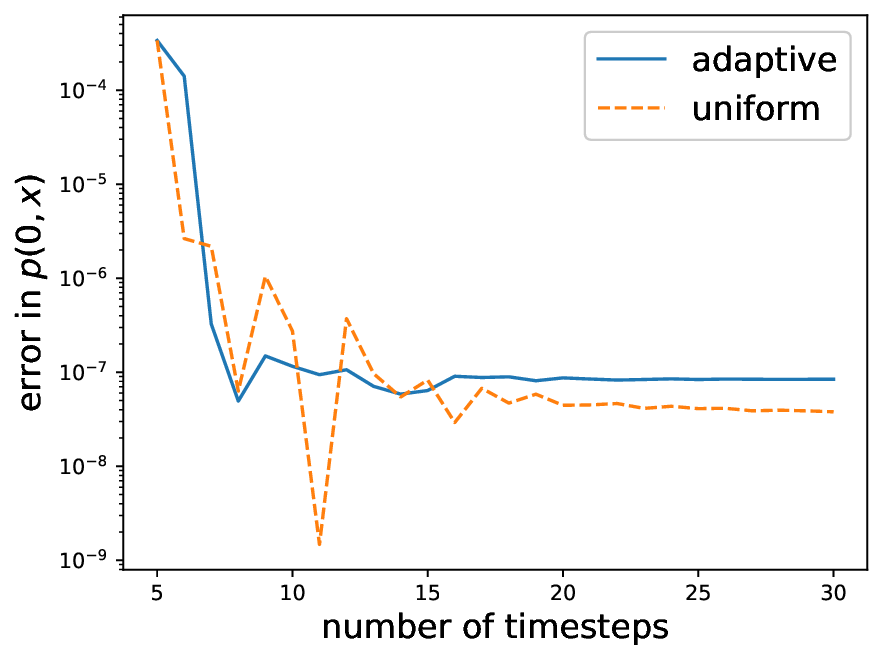}
            \caption{$\varepsilon = 0.5$}
        \end{subfigure}\\
        \caption{Comparison between the error in the adjoint's initial condition on a time adaptive grid and on a uniform grid. For a given number of steps in time $N$, the uniform grid is created with evenly distributed timesteps, while the adaptive grid evolved from a uniform grid with $N = 5$ timesteps using the adaptive cycle. The number of steps in space is constant for both grids with $d = 40$. The considered problem is described in Example~\ref{ex:eps-setup} with $\Omega_T = (0,1)^2$, $\nu = 0.1$, $\alpha=1$ and $m=0.5$. Each plot depicts the mean-square error between the initial state of the constructed adjoint $p$ and the initial state of the computed adjoint $p^k$ in dependence of the number of timesteps, i.e.\ $MSE=\frac{1}{d} \sum_{i=1}^{d} (p(0,x_i) -  p^k (0, x_i))^2$}
        \label{fig:eps-error}
    \end{figure}

    \begin{figure}
        \centering
        \begin{subfigure}[t]{0.3\textwidth}
            \centering
            \includegraphics[width=\linewidth]{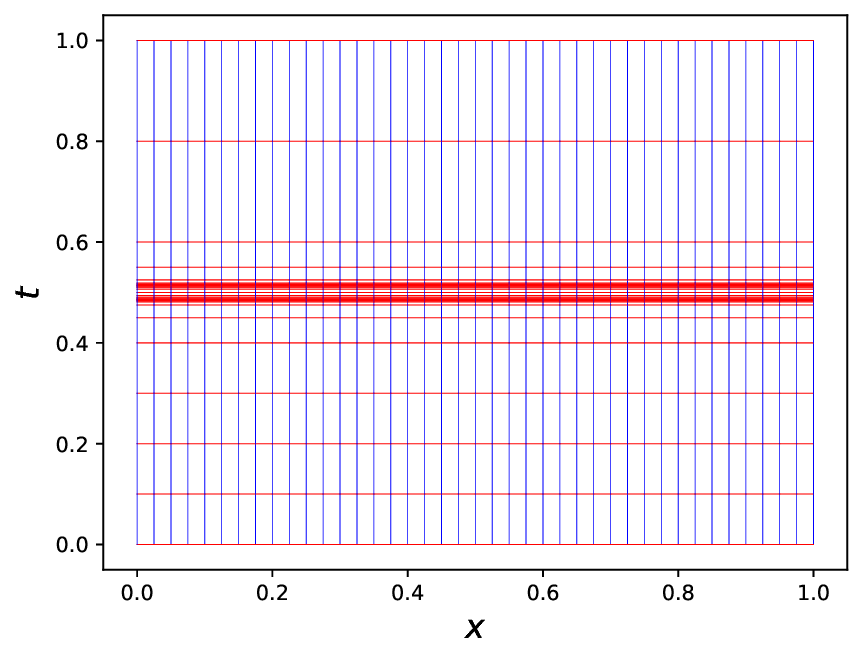}
            \caption{$\varepsilon = 0.05$}
        \end{subfigure}
        \hspace*{\fill}
        \begin{subfigure}[t]{0.3\textwidth}
            \centering
            \includegraphics[width=\linewidth]{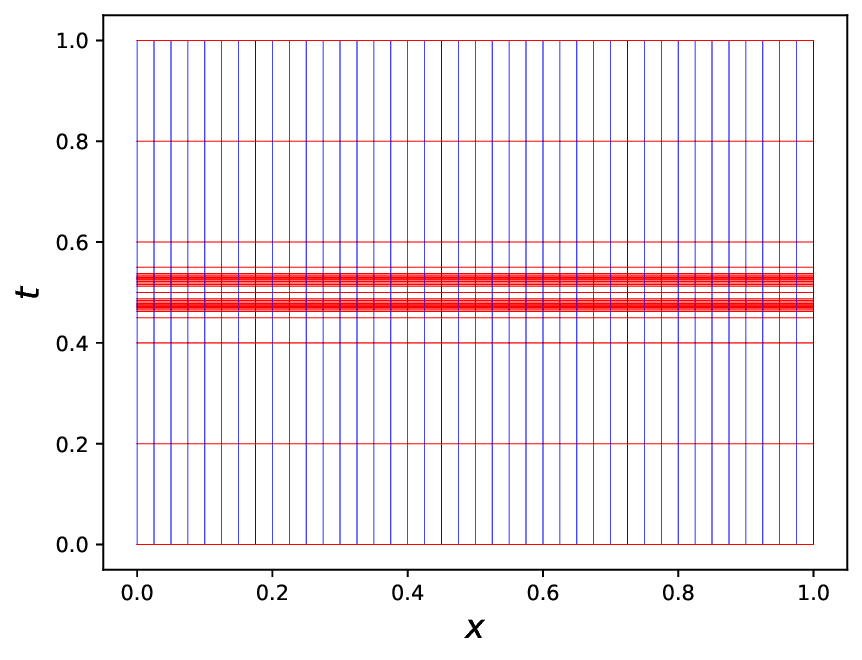}
            \caption{$\varepsilon = 0.1$}
        \end{subfigure}
        \hspace*{\fill}
        \begin{subfigure}[t]{0.3\textwidth}
            \centering
            \includegraphics[width=\linewidth]{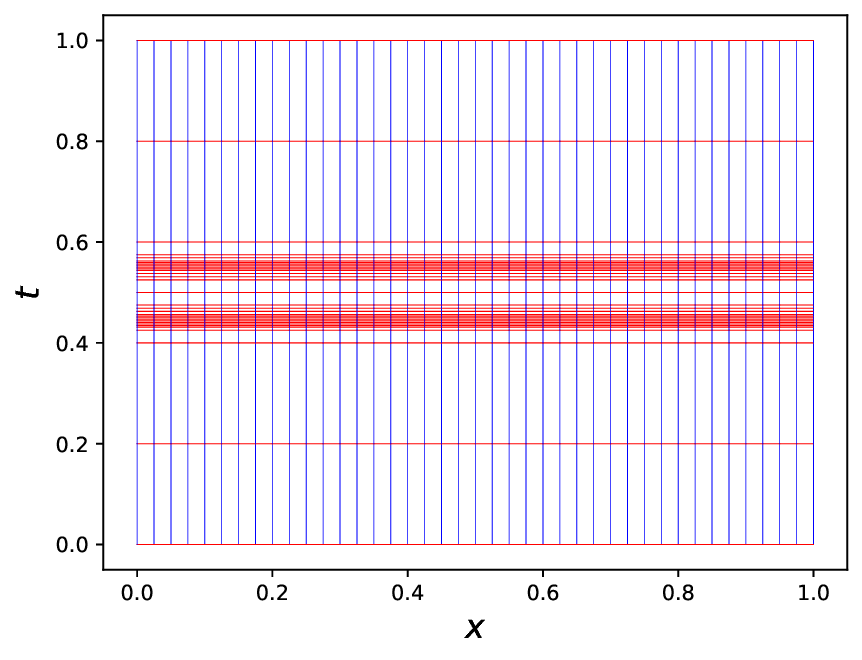}
            \caption{$\varepsilon = 0.2$}
        \end{subfigure}\\

        \centering
        \begin{subfigure}[t]{0.3\textwidth}
            \centering
            \includegraphics[width=\linewidth]{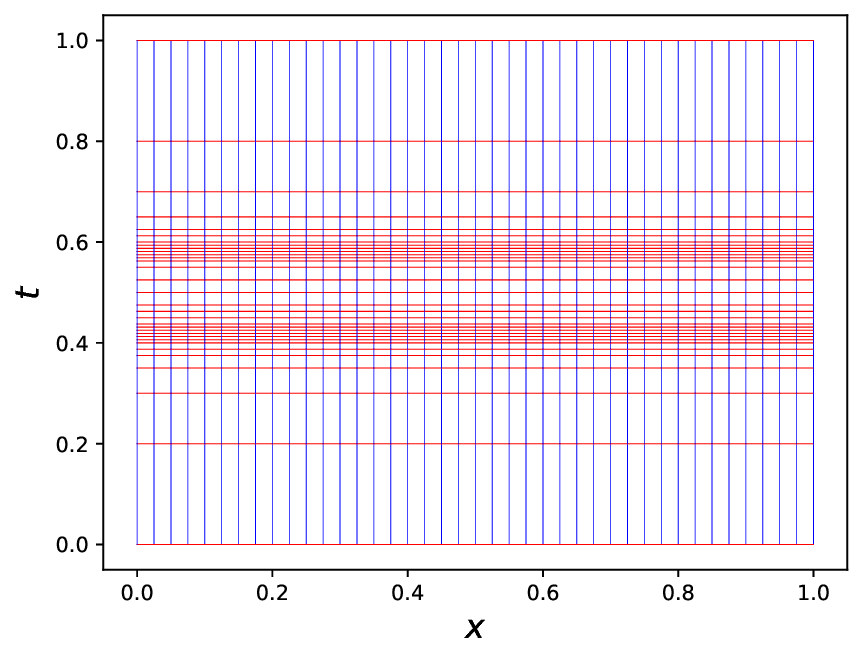}
            \caption{$\varepsilon = 0.3$}
        \end{subfigure}
        \hspace*{\fill}
        \begin{subfigure}[t]{0.3\textwidth}
            \centering
            \includegraphics[width=\linewidth]{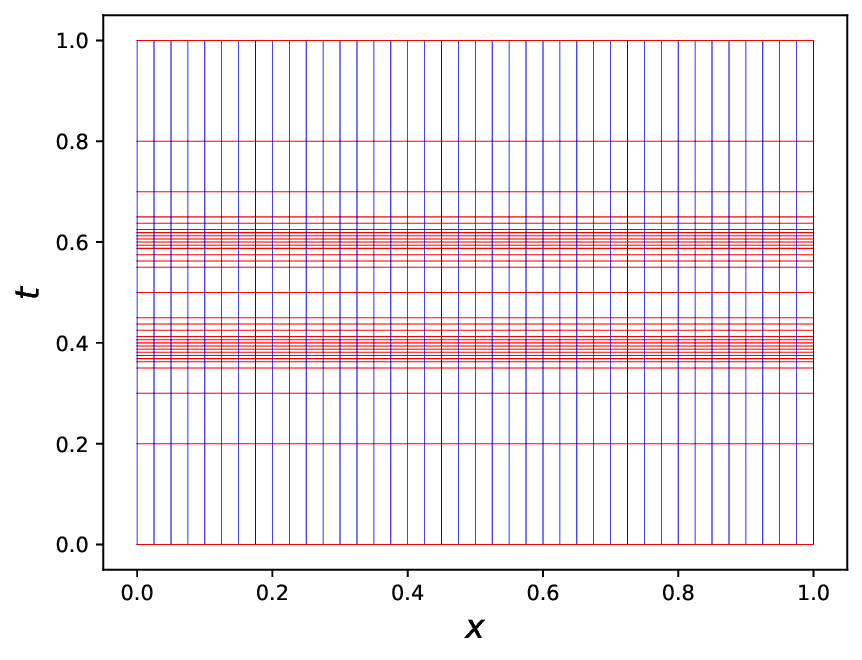}
            \caption{$\varepsilon = 0.4$}
        \end{subfigure}
        \hspace*{\fill}
        \begin{subfigure}[t]{0.3\textwidth}
            \centering
            \includegraphics[width=\linewidth]{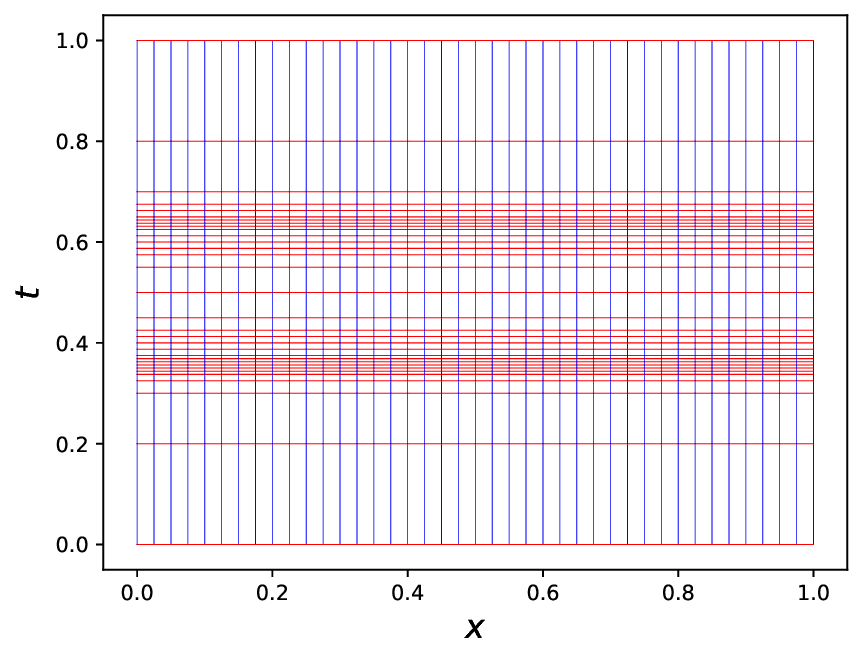}
            \caption{$\varepsilon = 0.5$}
        \end{subfigure}\\
        \caption{Structure of the time adaptive grid for the problem as described in Example~\ref{ex:eps-setup} with $\Omega_T = (0,1)^2$, $\nu = 0.1$, $\alpha=1$ and $m=0.5$. Each plot depicts an adaptive time grid for a different choice of $\varepsilon$. Each adaptive grid was created from a grid with $N=5$ evenly spaced timesteps by bisection of the interval with the largest error estimate. The adaptive cycle has been repeated $25$ times until $N=30$ timesteps have been reached. The number of steps in space is constant $d = 40$}
        \label{fig:eps-grid}
    \end{figure}

    \section{Conclusion and outlook}
    As a first result, we have successfully adapted the reformulation strategy from distributed control problems to initial control problems. Hence, we are now able to solve certain data assimilation problems, which are governed by a parabolic second order PDE, by the discretization of a PDE which is second order in time and fourth order in space. Furthermore, we demonstrated the functionality of this concept in numerical examples.

    As a second result, we came up with an error estimate for the initial condition of the adjoint state. Since the a-posteriori error bound is dependent on data from instances in time, we were able to decide at which time instances the discretization grid requires refinement. The resulting adaptive grid turned out to be better than initially anticipated, since it can be computed a-priori before the first computation of the adjoint state if linear finite elements in space and time are used for the discretization. \\

    Due to the novelty of the method and its significant difference from classical solution techniques for data assimilation, there are numerous aspects which form the basis for further research. We intend to compare the approach with classical methods for solving data assimilation problems and established methods for optimal control problems. Notice that this also requires to come up with suitable metrics on whose basis the comparison can be achieved. 
    
    Furthermore, we are interested in the extension of the method from the currently provided artificial environment to a setup with more realistic assumptions. This includes the restriction of the control to $\Uad$ and therefore the installment of the projection mapping as given in \eqref{eq:projected-variational}, lifting the regularity assumptions on the observations $y^{(d)}$ such that sparsely given observations may also be considered, as well as weighting the norms in the control problem \eqref{eq:control-problem} in order to include knowledge about the error from applications and previous executions (cf. Remark~\ref{rem:alpha}). Since solving the elliptic problem for more realistic scenarios can be expensive, we plan to study the application of model reduction techniques in order to decrease the computational complexity.

    Finally, let us mention that it would be beneficial to come up with an error estimation which allows to estimate the error both in space and time in order to surpass the heuristic assumption mentioned in Remark~\ref{rem:heuristic_assumption}.

    \bibliographystyle{scientific-computing}
    \bibliography{main}

\end{document}